\documentclass[19pt,a4paper]{amsart}
\usepackage{amscd,amssymb,amsopn,amsmath,amsthm,graphics,amsfonts,enumerate,verbatim,calc}
\usepackage[dvips]{graphicx}

\usepackage{amssymb,amsmath, color}

\headsep=1cm \numberwithin{equation}{section}
\newtheorem{theorem}{Theorem}[section]
\newtheorem{lemma}[theorem]{Lemma}
\newtheorem{notation}[theorem]{Notation}
\newtheorem{proposition}[theorem]{Proposition}
\newtheorem{corollary}[theorem]{Corollary}

\theoremstyle{definition}
\newtheorem{definition}[theorem]{Definition}
\theoremstyle{remark}
\newtheorem{remark}[theorem]{Remark}
\newtheorem{example}[theorem]{Example}
\newtheorem{discussion}[theorem]{Discussion}
\newtheorem{question}[theorem]{Question}

\newcommand{\im}{\operatorname{im}}

\newcommand{\pgrade}{\operatorname{p.grade}}
\newcommand{\Kgrade}{\operatorname{K.grade}}

\newcommand{\Spec}{\operatorname{Spec}}

\newcommand{\rad}{\operatorname{rad}}
\newcommand{\cd}{\operatorname{cd}}

\newcommand{\Ht}{\operatorname{ht}}
\newcommand{\pd}{\operatorname{p.dim}}

\newcommand{\V}{\operatorname{V}}
\newcommand{\Var}{\operatorname{Var}}

\newcommand{\Ext}{\operatorname{Ext}}

\newcommand{\Tor}{\operatorname{Tor}}

\newcommand{\Hom}{\operatorname{Hom}}

\newcommand{\conv}{\operatorname{Conv}}

\newcommand{\fm}{\frak{m}}
\newcommand{\fp}{\frak{p}}
\newcommand{\fq}{\frak{q}}
\newcommand{\fa}{\frak{a}}

\begin{document}

\author[Asgharzadeh and Dorreh]{Mohsen Asgharzadeh and  Mehdi Dorreh }

\address{M. Asgharzadeh, School of Mathematics, Institute for Research in Fundamental
Sciences (IPM), P. O. Box 19395-5746, Tehran, Iran.}
\email{asgharzadeh@ipm.ir}
\address{M. Dorreh,  Department of Mathematics Shahid Beheshti
University Tehran, Iran.}\email{mdorreh@ipm.ir}

\title[Cohen-Macaulayness of non-affine normal semigroups]
{Cohen-Macaulayness of non-affine normal semigroups}

\keywords{ \v{C}ech cohomology; Cohen-Macaulay ring; cohomological dimension; convex bodies; direct limit; infinite-dimensional; invariant theory;  non-noetherian ring; monomial ideals; normal semigroup; Taylor resolution;   pigeonhole principle; projective dimension; quasi rational cone; vanishing theorems.\\The  research  was  supported by a grant from IPM, no. 91130407.}

\begin{abstract}
In this paper, we  study the Cohen-Macaulayness of non-affine normal semigroups in $\mathbb{Z}^n$. We do this by establishing the following four statements each
of independent interest: 1)  a Lazard type result  on $I$-supported elements of $\prod _{\mathbb{N}}\mathbb{Q}_{\geq0}$ for an index set $I\subset\mathbb{N}$; 2)  a criterion of regularity of sequences of elements of the ring via projective dimension; 3) a direct limit of polynomial rings with toric maps; 4) any direct summand of rings of the third item is Cohen-Macaulay. To illustrate the idea,
we give many examples.
\end{abstract}

\maketitle

\section{Introduction}

By Gordan's Lemma any affine and normal semigroup comes from the lattice points
of a finitely generated rational cone.
One of our interest in this paper is to understand homological properties of the lattice points of the intersection of  half spaces. 
Half spaces are not  necessarily closed
and the number of half spaces is not necessarily finite.
The intersection of two open half spaces can be represented by a direct union of rational cones. This
gives a direct union of Cohen-Macaulay affine algebras. In view of Remark \ref{adt}, Cohen-Macaulayness is not closed under taking the direct limit.
In this paper we prove the following result.

\begin{theorem}(see Theorem \ref{main}) Let $k$ be a field and $H\subseteq \mathbb{Z}^\infty$ be a normal semigroup (not necessarily affine). Then any
monomial strong parameter sequence of
$k[H]$ is regular.
\end{theorem}

Here,  $\mathbb{Z}^\infty$ is $\bigcup_{s\in \mathbb{N}} \mathbb{Z}^s$. Also, $k[H]:=\bigoplus_{h\in H} k X^h$ is the $k$-vector space $\bigoplus_{h\in H} k X^h$. It carries a natural multiplication whose table is given by  $X^h X^{h^\prime}:= X^{h+h^\prime}$. Note that $k[H]$ is equipped with  a structure of $H$-graded ring defined by the semigroup $H$. By \textit{monomial},  we mean homogenous elements of $k[H]$. The concept of \textit{strong parameter sequence}  is a non-noetherian  version of  system of parameters in the local algebra. It is introduced
in \cite{HM}; see Definition \ref{defhm}.

Theorem 1.1 drops two finiteness assumptions
of \cite[Theorem 1]{H1}. The proof involves on the notion of toric maps and the notion of full extensions of semigroups.  Such a semigroup extension is the set
of solutions of a system of homogeneous linear equations with integer coefficients.  These kind of semigroups appear in many contexts. We refer the reader to \cite[Theorem 9.2.9]{C}, to deduce \cite[Theorem 1]{H1} as a consequence of Batyrev-Borisov vanishing Theorem.
Its proof uses many things. All of them involved on certain finiteness conditions. For further references please see \cite{M} and \cite{GWii}.

In Section 8, as an example in practise, we study the   Cohen-Macaulayness of semigroups that arise from quasi ration plane cones.  We prove this by introducing the following four classes of semigroups:

\begin{enumerate}
\item[(i)]$ H := \{(a,b) \in \mathbb{N}_0^2| 0 \leq b/a < \infty\} \cup \{(0,0)\}$;
\item[(ii)] $ H^{\prime} := \{(a,b) \in \mathbb{N}^2| 0 < b/a < \infty\} \cup \{(0,0)\}$;
\item[(iii)]$ H_1 := \{(a,b) \in \mathbb{Z}^2|  b \in \mathbb{N}_0,\; if \;a \; is \; negative\; b\neq 0\} \cup \{(0,0)\}$;
    \item[(iv)]$ H_2 := \{(a,b) \in \mathbb{Z}^2|  b \in \mathbb{N}\} \cup \{(0,0)\}$.
\end{enumerate}
For more details; see Theorem \ref{ggg}.

Non-affine semigroups appear naturally in the study of the Grothendieck ring of varieties over a field $k$. To clarify this,
let $SB$ denote the set of stably birational equivalence classes of irreducible
algebraic varieties over $k$. Then by \cite{LL}, the Grothendieck ring of varieties mod to a line is $\mathbb{Z}[SB]$.
For an application of non-affine semigroups on affine semigroups, we recommend the reader to see \cite{G}.

Throughout this paper, $R$ denote a commutative ring with identity and all modules are assumed to be left unitary.
We refer the reader to  the books \cite{BH}  and \cite{BG} for all unexplained definitions in the sequel.

\section{outline of the proof}

We give an outline of the proof of Theorem 1.1.
The proof easily reduces to the case that $H\subseteq \mathbb{Z}^n$ is normal (but not affine), see Lemma \ref{wpdhm}.
 Similar to the  affine case, we assume that
 $H\subseteq \mathbb{Z}^n$ is \textit{positive}, see
 Lemma \ref{ho}. Recall that $H$ is positive if there is not any invertible element in $H$.

 \begin{notation}Denote the set of all
 nonnegative rational numbers by $\mathbb{Q}_{\geq0}$.
  \end{notation}

The first step is to understand the structure of  $\prod_{\mathbb{N}}\mathbb{Q}$ as a semigroup.
To state it, let $I\subseteq\mathbb{N}$ be an infinite index set.
Denote the i-th component of $\alpha\in\prod_{\mathbb{N}}\mathbb{Q}$ by $\alpha_i$. Then
$\alpha$ is called $I$-\textit{supported} if  $\alpha_i \neq 0$ for all $i\in I$.
Also, an element $\alpha\in\prod_{\mathbb{N}}\mathbb{Q}$  is called \textit{almost non-negative}, if there exists only  finitely many $i \in \mathbb{N}$ such that $\alpha_i$ is negative.

\begin{lemma}(see Corollary \ref{maincor})
Let $M\subseteq\prod_{i \in \mathbb{N}} \mathbb{Q}$ be the set of all almost non-negative  and $I$-supported elements.
Then the semigroup  $\widetilde{H}:=M\cup\{0\}$   is a direct limit of $\{\mathbb{Q}_{\geq0}^{n_i}:i\in J\}$.
\end{lemma}

This
 Lazard-type result implies the following.

\begin{lemma}(see Theorem \ref{prod}) Let $k$ be a field and $H\subseteq \mathbb{Z}^n$ be a positive and normal semigroup.  Then there is a direct system $\{A_n:n\in \mathbb{N}\}$  with the following properties:
\begin{enumerate}
\item[(i)] $A_n$ is a  Noetherian polynomial ring over $k$   for all  $n\in \mathbb{N}$.
\item[(ii)] $A_n \to A_m$ is toric for all $n\leq m$.
\item[(iii)] $k[H]$ is a direct summand of ${\varinjlim}_{n\in \mathbb{N}} A_n$.
\end{enumerate}
\end{lemma}

Thus, we look at the following question.

\begin{question}\label{q1}
Let  $\{A_\gamma:\gamma\in \Gamma\}$  be a direct family of Noetherian regular rings and let $R$ be a direct summand of
$A= {\varinjlim}_{\gamma\in \Gamma} A_\gamma$. Is $R$  Cohen-Macaulay?
\end{question}

Note that Cohen-Macaulayness is not closed under taking direct limit, see Remark \ref{adt}.
Also,  Remark \ref{adt} provides a reason for working with Cohen-Macaulayness in the sense of \cite{HM}.
We explain our method to handle Question \ref{q1}.
First, we  give the following auxiliary result.

\begin{lemma}(see Theorem \ref{pd2})
Let  $A$  be a Noetherian polynomial ring over a field  and $\underline{x}:=x_1,\ldots,x_n$ a monomial sequence in $A$.
If $\pd(A/(x_{i_1},\ldots, x_{i_k})A)=k$ for all $1\leq i_1<\ldots <i_k\leq n$, then  $\underline{x}$
is a regular sequence in $A$.
\end{lemma}

Then by applying the above lemma, Theorem 1.1 follows easily by the following Lemma.

\begin{lemma}(see Theorem \ref{pro})
Let  $\{A_\gamma:\gamma\in \Gamma\}$  be a direct family of Noetherian polynomial rings over a field with toric maps and let $R$ be a direct summand of $A:= {\varinjlim}_{\gamma\in \Gamma} A_\gamma$.
Let $\underline{x}:=x_1,\ldots,x_\ell$ be a monomial strong parameter sequence in $R$. Then  $\underline{x}$
is a regular sequence in $R$.
\end{lemma}

\section{A Lazard type Result}

The main results of this section  are Lemma \ref{mkey} and Corollary \ref{maincor}. They have essential role in the next section.
Recall that $\mathbb{Q}_{\geq0}$ is the set of all
 nonnegative rational numbers.
Our initial aim is to understand the structure of $\prod \mathbb{Q}_{\geq0}$ as a semigroup.
Note that $\prod  \mathbb{Q}$ is a vector space over $\mathbb{Q}$. Let $J$ be a base for it. Then
$\prod\mathbb{Q}\cong(\bigoplus_{J}\mathbb{Q})$ as $\mathbb{Q}$-vector spaces. But,  this
isomorphism does not send $\prod \mathbb{Q}_{\geq0}$ to $\bigoplus_{J}\mathbb{Q}_{\geq0}$.
Note that we are in the context of semigroups. The use of the minus is the main difficulty.
We recommend the reader to see \cite{B} for comparison-type results between product and coproduct.

\begin{definition}
Let $(x_n)_{n\in \mathbb{N}}$ and $(y_n)_{n\in \mathbb{N}}$ be two sequences of rational numbers.
We say $(x_n)_{n\in \mathbb{N}}\leq(y_n)_{n\in \mathbb{N}}$, if $x_n\leq y_n$ for all $n$.
Also, we denote the i-th component of the sequences
$(x)_{i\in\mathbb{N}}$ and  $\alpha$ by $(x)_i$  and  $\alpha_i$, respectively.
\end{definition}

\begin{definition}
Let $I\subseteq\mathbb{N}$ be an infinite index set.
An element
$\alpha\in\prod_{i \in \mathbb{N}} \mathbb{Q}$ is called $I$-\textit{supported} if  $\alpha_i \neq 0$ for all $i\in I$.
\end{definition}
 When we refer to an $I$-\textit{supported} element, we adopt that $I$ is infinite.

\begin{lemma}\label{ffkey}
Let $\{\beta_1,\ldots, \beta_n\}$ be a set of $I$-supported elements of $\prod_{i \in \mathbb{N}} \mathbb{Q}_{\geq0} $ that are linearly independent over $\mathbb{Q}$ and let  $\alpha\in\prod_{i \in \mathbb{N}} \mathbb{Q}_{\geq0}$ be $I$-supported.
Suppose $$\alpha = \beta_m + \ldots+ \beta_{n-1} - \beta_n .\  \ (\ast)$$
Then there is $(\beta _m)^{\prime}\in\prod_{i \in \mathbb{N}} \mathbb{Q}_{\geq0}$ such that  the following $I$-supported set $$\Gamma_m:=\{\beta_1,\ldots, \beta_{m-1},\beta_m - (\beta _m)^{\prime},\beta_{m+1} , \ldots ,\beta_{n-1},\beta_n - (\beta _m)^{\prime},(\beta _m)^{\prime} \}\subseteq \prod_{i \in \mathbb{N}} \mathbb{Q}_{\geq0}$$ is linearly independent over $\mathbb{Q}$. In particular,  $\{\beta_1,\ldots, \beta_n\} \subseteq \sum_{\gamma\in\Gamma_m}\mathbb{Q}_{\geq0}\gamma$.
\end{lemma}

\begin{proof}
For each $i\in I$, we look at a set $\{{(\beta_m)_i}^{\prime},\ldots,{(\beta_{n-1})_i}^{\prime}\}$ of positive rational numbers  with the following properties
\begin{enumerate}
\item[(3.3.1)] ${(\beta_k)_i}^{\prime} < {(\beta_k)_i}$ for all $m \leq k \leq n-1$;
\item[(3.3.2)] ${(\beta_m)_i}^{\prime}+ \ldots +{(\beta_{n-1})_i}^{\prime}= (\beta_n)_i$.
\end{enumerate}
Such a thing exists, because  $(\beta_n)_i< {(\beta_m)_i}+ \ldots +{(\beta_{n-1})_i}$ and all of these are positive by $(\ast)$. For each $m \leq k \leq n-1$, we bring the following claim.

\begin{enumerate}
\item[] \textbf{Claim. }There are infinitely  many ways to choose ${(\beta_k)_i}^{\prime}$.
\item[] Indeed, we clarify this for $(\beta_m)_i^{\prime}$.
It is enough to replace $(\beta_m)_i^{\prime}$ by $(\beta_m)_i^{\prime}\pm1/ \ell$ and $(\beta_{m+1})_i^{\prime}$ by $(\beta_{m+1})_i^{\prime}\mp1/ \ell$ for
all sufficiently large $\ell\in \mathbb{N}$. Note that $(\beta_{m+1})_i^{\prime}\mp1/ \ell$ and $(\beta_m)_i^{\prime}\pm1/ \ell$ are positive, because   $(\beta_{m+1})_i,(\beta_m)_i>0$.
\end{enumerate}

Let $m \leq k \leq n-1$. We want to  define  $\beta_k^{\prime}$ in the reminding components.   Take $i\in \mathbb{N}\setminus I$ and suppose $0 < (\beta_n)_i$. Then there are nonnegative (not necessarily positive) rational numbers $$\{{(\beta_m)_i}^{\prime},\ldots,{(\beta_{n-1})_i}^{\prime}\} \ \ (\dag)$$ with the following two properties\\

\begin{enumerate}
\item[]  $\bullet: {(\beta_k)_i}^{\prime}\leq {(\beta_k)_i}$ (not necessarily strict inequality), and \\ $\bullet: {(\beta_m)_i}^{\prime}+ \ldots +{(\beta_{n-1})_i}^{\prime}= (\beta_n)_i$ .

\end{enumerate}

Note that there exists at least one choice for ${(\beta_k)_i}^{\prime}$. Define $$(\beta_k)_i^{\prime} :=\Bigg\{ ^{(\beta_k)_i
\  \  \  \  \   \  \  \  \   \  \    \    \    \  \  \  \    \    \    \  \  \     \textit{ if  }\    \ (\beta_n)_i =0}
_{\textit{ define by  }(\dag)  \   \    \    \   \  \  \  \  \  \  \  \   \  \  if \    \ (\beta_n)_i \neq0.}$$

Keep in mind  the above Claim, $|I|=\infty$ and that  $ \prod_{\mathbb{N}}\mathbb{N}$ is uncountable. These turn out that
there are uncountably many ways to choose the sequence $(\beta _m)^{\prime}:=({(\beta_m)_i}^{\prime})_{i \in \mathbb{N}}$.  We pick one of them with the following property
 $$(\beta _m)^{\prime} \notin \mathbb{Q}\beta_1 + \ldots + \mathbb{Q}\beta_n.$$
We can take such a sequence, because  $\mathbb{Q}\beta_1 + \ldots + \mathbb{Q}\beta_n$
  is countable.

Look at  $\Gamma_m:= \{\beta_1,\ldots, \beta_{m-1},\beta_m - (\beta _m)^{\prime},\beta_{m+1} , \ldots ,\beta_{n-1},\beta_n - (\beta _m)^{\prime},(\beta _m)^{\prime} \}.$ Clearly, $\{\beta_1,\ldots, \beta_n\} \subseteq \sum_{\gamma\in\Gamma_m}\mathbb{Q}_{\geq0}\gamma$, and so $$n\leq\dim_{\mathbb{Q}} \mathbb{Q}\Gamma_m\leq |\Gamma_m|=n+1.$$ Since $(\beta _m)^{\prime}\notin \mathbb{Q}\{\beta_1,\ldots, \beta_n\}$, $\dim_{\mathbb{Q}}  \mathbb{Q}\Gamma_m = n+1$. That is $\Gamma_m$ is linearly independent over $\mathbb{Q}$.  Let $i \in I$. Then $\gamma_i>0$ for all  $\gamma \in \Gamma_m$. This finishes the proof.
\end{proof}

The following is our key lemma. We prove it by using the reasoning of Lemma \ref{ffkey} several times.

\begin{lemma}\label{fkey}
Let $\{\beta_1,\ldots, \beta_n\}$ be a set of $I$-supported elements of $\prod_{i \in \mathbb{N}} \mathbb{Q}_{\geq0} $ that are linearly independent over $\mathbb{Q}$ and  let $\alpha\in\prod_{i \in \mathbb{N}} \mathbb{Q}_{\geq0}$ be $I$-supported.
Suppose $$\alpha = \eta_1\beta_1 +\ldots +\eta_{n-1}\beta_{n-1} - \beta_n \  \ (\dag)$$  where $\eta_i\in\{0,1\}$ for all $1\leq i \leq n-1$.
Then there exists a finite set $\Gamma \subseteq \prod_{i \in \mathbb{N}} \mathbb{Q}_{\geq0}$ of $I$-supported elements, linearly independent over $\mathbb{Q}$ and $\{\alpha,\beta_1,\ldots, \beta_n\} \subseteq  \sum_{\gamma\in \Gamma}\mathbb{Q}_{\geq0}\gamma$.
\end{lemma}

\begin{proof}
First assume that $\alpha = \beta_ k - \beta_n$ for some $1 \leq k \leq n-1$. Then $\{\beta_1, \ldots,\widehat{\beta_k},\ldots ,\beta_{n}, \alpha\}$ is the desirable set.

Fix $ 1 < m <n-1$ and assume that  $\eta_i =0$ for all  $1 \leq i \leq m-1$ and that $\eta_i =1$ for all $ m \leq i \leq n-1$. Then,    $\alpha = \beta_m + \ldots+ \beta_{n-1} - \beta_n$. Let $(\beta _m)^{\prime}$ be as Lemma \ref{ffkey}.
Put $$\alpha_m := \beta_{m+1} + \ldots+ \beta_{n-1}- (\beta_n -{\beta_m}^{\prime}).$$ In view of  (3.3.1) and (3.3.2), $\alpha_m$  is $I$-supported. Define $\Gamma_m$ by the Lemma \ref{ffkey} and apply Lemma \ref{ffkey} for the new data $\{\Gamma_m,\alpha_m\}$. By repeating this procedure, we find an $I$-supported element $(\beta _{n-3})^{\prime}\in\prod \mathbb{Q}_{\geq0}\setminus \mathbb{Q}\Gamma_ {n-4}$ and  the following $I$-supported set
\[\begin{array}{ll}
\Gamma_{n-3}:=&\{\beta_1,\ldots,\beta_{m-1}\}\cup\\
& \{\beta_m - (\beta _m)^{\prime}, \ldots ,\beta_{n-3}-(\beta _{n-3})^{\prime}\}\cup\\
& \{\beta_{n-2},\beta_{n-1},\beta_n - \sum_{i=m}^{n-3}(\beta _i)^{\prime} \}\cup\\
&\{(\beta _m)^{\prime}, \ldots , (\beta _{n-3})^{\prime}\}
\end{array}\]
such that $\dim \mathbb{Q}\Gamma_{n-3}=2n-m-2$  and $\{\beta_1,\ldots, \beta_n\}\subseteq \sum_{\gamma\in \Gamma_{n-3}}\mathbb{Q}_{\geq0}\gamma$.
Look at  $$\alpha_{n-3} := \beta_{n-2} + \beta_{n-1}- (\beta_n- {\beta_m}^{\prime}-\ldots- {\beta_{n-3}}^{\prime}).$$ Then $\alpha_{n-3}$ is $I$-supported.
 Fix the data $(\alpha_{n-3},\Gamma_{n-3})$ and apply the reasoning of Lemma \ref{ffkey} to find $I$-supported elements $(\beta _{n-2})^{\prime}, (\beta _{n-1})^{\prime}$  with the following properties:\begin{enumerate}
\item[(1)] $(\beta _{n-2})^{\prime} + (\beta _{n-1})^{\prime} = \beta_n - ((\beta _m)^{\prime}+ \ldots + (\beta _{n-3})^{\prime})$;
\item[(2)]  $\beta_{n-1}-(\beta _{n-1})^{\prime},\beta_{n-2}-(\beta _{n-2})^{\prime} \in  \prod \mathbb{Q}_{\geq0}$ are $I$-supported;
\item[(3)]  $\beta _{n-2}^{\prime}\notin\mathbb{Q}\Gamma_{n-3}$ and  $\beta _{n-1}^{\prime}\notin\mathbb{Q}\Gamma_{n-3}\cup\mathbb{Q}\beta _{n-2}^{\prime}$.
\end{enumerate}

Now we define $$\Gamma_{n-2}:= \{\beta_1, \ldots,\beta_{m-1},\beta_m - (\beta _m)^{\prime}, \ldots ,\beta_{n-1}-(\beta _{n-1})^{\prime}, (\beta _m)^{\prime}, \ldots , (\beta _{n-1})^{\prime} \}.$$  It has the following properties:

\begin{enumerate}
\item[(i)] $\dim \mathbb{Q}\Gamma_{n-2}=|\Gamma_{n-2}|=2n-m-1$, by (3);
\item[(ii)] $0 < \gamma_i$ for all $i \in I$ and $\gamma \in \Gamma$, by (2);
\item[(iii)] $\{\beta_1,\ldots, \beta_n,\alpha \} \subseteq \sum_{\gamma\in \Gamma_{n-2}} \mathbb{Q}_{\geq0}\gamma$, because
\end{enumerate}
\[\begin{array}{ll}
\alpha &= \sum_{i=m}^{n-1}\beta_i - \beta_n\\&=\sum_{i=m}^{n-1}(\beta_i-(\beta _i)^{\prime})+ \sum_{i=m}^{n-1}(\beta _i)^{\prime}- \beta_n\\
&\stackrel{(1)}=\sum_{i=m}^{n-1}(\beta_i-(\beta _i)^{\prime})+ \beta_n- \beta_n\\
&\in\sum_{\gamma\in \Gamma_{n-2}} \mathbb{Q}_{\geq0}\gamma.
\end{array}\]
It is now clear that $\Gamma:=\Gamma_{n-2}$ is the set that we search for it.
\end{proof}

Our next aim is to drop the assumption  $(\dag)$ of Lemma \ref{fkey}.

\begin{lemma}\label{key}
Let $\{\beta_1,\ldots, \beta_n\}$ be a set   of $I$-supported elements of $\prod_{i \in \mathbb{N}} \mathbb{Q}_{\geq0} $ that are linearly independent over $\mathbb{Q}$ and  $\alpha\in\prod_{i \in \mathbb{N}} \mathbb{Q}_{\geq0}$ be nonzero.
Then there exists  a finite set   of $I$-supported elements $\Gamma \subseteq \prod_{i \in \mathbb{N}} \mathbb{Q}_{\geq0}$ such that $\Gamma$ is linearly independent over $\mathbb{Q}$ and $\{\alpha,\beta_1,\ldots, \beta_n\} \subseteq  \sum_{\gamma\in \Gamma}\mathbb{Q}_{\geq0}\gamma$.
\end{lemma}

\begin{proof}
 If $\dim (\mathbb{Q}\beta_1 + \ldots + \mathbb{Q}\beta_n + \mathbb{Q}\alpha) = n+1$, there is no thing to prove. Thus,
 we can assume that $\dim (\mathbb{Q}\beta_1 + \ldots + \mathbb{Q}\beta_n + \mathbb{Q}\alpha) = n.$ There are positive rational numbers $\{\epsilon_{\ell}\}$ such that
 $\alpha = \sum_{i=m}^{n-j}\epsilon_i\beta_i -\sum_{i=n-j+1}^{n}\epsilon_i\beta_i.$ If $m = n-j$, the set $\{\beta_1,\ldots,\beta_{m-1},\beta_{m+1},\ldots,\beta_n,\alpha\}$ is the desirable set. So we assume that $m < n-j$. Also, without loss of the generality, we assume that $$\alpha =\beta_m +\ldots +\beta_{n-j} -\beta_{n-j+1}-\ldots-\beta_n.$$
We argue by induction on $j$.  Lemma \ref{fkey} yields the proof when $j=1$. Now suppose $j>1$ and assume inductively that the result has been proved for $j-1$. Put $$\alpha_1 := \beta_m +\ldots+\beta_{n-j}- \beta_{n-j+1}.$$ Then by Lemma \ref{fkey}, there exists  a finite set $\Gamma_1$ of $I$-supported
elements that  are linearly independent over $\mathbb{Q}$ and $$\{\alpha_1,\beta_1,\ldots, \beta_{n-j+1}\} \subseteq  \sum_{\gamma\in \Gamma_1}\mathbb{Q}_{\geq0}\gamma.$$ By replacing  $\{\gamma_1,\ldots,\gamma_{\ell}\}$ with a suitable scaler multiplication, we have $\alpha_1=\gamma_1+\ldots+\gamma_{\ell}$ for some $\{\gamma_i\}$.
By the reasoning of Lemma \ref{fkey}, we have uncountable choice for each elements of $\Gamma_1$. Hence, we can choose $\Gamma_1$ such that $$\Gamma_2 := \Gamma_1  \cup \{ \beta_{n-j+2},\ldots,\beta_n\}$$ is linearly independent over $\mathbb{Q}$. Rewrite $$\alpha = \gamma_1+\ldots+\gamma_{\ell} - \beta_{n-j+2}-\ldots -\beta_n.$$
The number of negative signs appear in this presentation is $j-1$. To finish the proof, it remains to apply the induction hypothesis.
\end{proof}

\begin{corollary}\label{hkey}
Let $\{\beta_1,\ldots, \beta_n\}$ be a  set   of $I$-supported elements of $\prod_{i \in \mathbb{N}} \mathbb{Q}_{\geq0} $.
Then there exists  a finite set $\Gamma \subseteq \prod_{i \in \mathbb{N}} \mathbb{Q}_{\geq0}$  of $I$-supported elements    such that $\Gamma$ is linearly independent over $\mathbb{Q}$ and $\{\beta_1,\ldots, \beta_n\} \subseteq  \sum_{\gamma\in \Gamma}\mathbb{Q}_{\geq0}\gamma$.
\end{corollary}

\begin{proof}
We use induction on $n$. When $n=1$, there is nothing to prove. By induction hypothesis, there exists  a finite set $\Delta \subseteq \prod_{i \in \mathbb{N}} \mathbb{Q}_{\geq0}$ such that $\Delta$ is linearly independent over $\mathbb{Q}$, $\{\beta_1,\ldots, \beta_{n-1}\} \subseteq  \sum_{\delta\in \Delta}\mathbb{Q}_{\geq0}\delta $ and $\delta_i \neq 0$ for all $i \in I$. Applying Lemma \ref{key} for $\Delta$ and $\beta_n$, yields the claim.
 \end{proof}

\begin{definition}
An element $\alpha\in\prod_{\mathbb{N}}\mathbb{Q}$  is called \textit{almost non-negative}, if there exists finitely many $i \in \mathbb{N}$ such that $(\alpha)_i$ is negative. An almost non-negative subset of  $\prod_{i \in \mathbb{N}} \mathbb{Q}$ can be defined in a similar way. An element $\beta\in\prod_{\mathbb{N}}\mathbb{Q}$  is called \textit{almost zero}, if there exists finitely many $i \in \mathbb{N}$ such that $\beta_i$ is nonzero.
\end{definition}

Now we are ready to prove the following.

\begin{lemma}\label{mkey}
Let $M\subseteq\prod_{i \in \mathbb{N}} \mathbb{Q}$ be  almost non-negative  and let $\{\beta_1,\ldots, \beta_n\}$ be a subset  of $I$-supported elements  of $M $ linearly independent over $\mathbb{Q}$ and  let $\alpha\in M$ be  $I$-supported.
Then there exists a finite set $\Gamma \subseteq M$  of $I$-supported elements such that $\Gamma$ is linearly independent over $\mathbb{Q}$ and $\{\alpha,\beta_1,\ldots, \beta_n\} \subseteq  \sum_{\gamma\in \Gamma}\mathbb{Q}_{\geq0}\gamma$.
\end{lemma}

\begin{proof}
Without loss  of the generality, we can assume that $$\alpha = \beta_m + \ldots + \beta_{n-j}- \beta_{n-j+1}-\ldots-\beta_{n}$$ and that $m < n-j$. The reason presented in Lemma \ref{key}. Let $1 \leq k \leq s$.
Take  the integer $l$ be such that  $\alpha_i$ and $(\beta_k)_i$ are nonnegative for all $i>l$. Define\\

$(\dot{\beta_k})_i:= \Bigg\{ ^{(\beta_k)_i \   \ \textit{for all} \ \ i \leq l }
_{0   \ \ \ \   \  \ \textit{for  all }\ \ i > l}$  \  \  \ \ \ \ and \ \ \ \ \ \ \ \ \ \ \ \ $(\ddot{\beta_k})_i:= \Bigg\{ ^{0 \   \ \ \   \   \textit{for  all} \ \ i \leq l }
_{(\beta_k)_i   \ \  \textit{for  all} \ \ i > l.}$\\Also, \\

 $(\dot{\alpha})_i:= \Bigg\{ ^{\alpha_i \  \ \ \ \textit{for all} \ \ i \leq l }
_{0   \ \ \ \   \  \ \textit{for all} \ \ i > l}$  \  \  \ \ \ \ \ \ and \ \ \ \ \ \ \ \ \ \ \ \ \ $(\ddot{\alpha})_i:= \Bigg\{ ^{0 \  \ \ \ \textit{for  all} \ \ i \leq l }
_{\alpha_i   \ \  \textit{for all} \ \ i > l.}$\\

Thus $\ddot{\alpha}$ and $\ddot{\beta_k}$ belong to $\prod_{i \in \mathbb{N}} \mathbb{Q}_{\geq0}$. Note that $\alpha=\ddot{\alpha}+\dot{\alpha}$. The same thing holds for $\beta_i$. Also, $\dot{\alpha}$ and   $\dot{\beta_k}$ are almost zero. The data $\{\ddot{\alpha} ,\ddot{ \beta_1}, \ldots , \ddot{ \beta_n}\}$ satisfies in the situation of Corollary \ref {hkey}. So, there exists  a set $\{\gamma_1,\ldots,\gamma_s\} \subseteq \prod_{i \in \mathbb{N}} \mathbb{Q}_{\geq0} $ of $I$-supported and linearly independent elements over $\mathbb{Q}$  with the property  $$\{\ddot{\alpha} ,\ddot{ \beta_1}, \ldots , \ddot{ \beta_n}\} \subseteq \mathbb{Q}_{\geq0} \gamma_1 + \ldots +   \mathbb{Q}_{\geq0} \gamma_s.$$ Furthermore, we can choose $\{\gamma_1,\ldots,\gamma_s\}$ such that $(\gamma_k)_i= 0$ for all $i \leq l$.  Due to $m < n-j$,  one gets $n+1 \leq s$. For each
$ 1 \leq t \leq n$, look at

$$ {\ddot{\alpha} = \sum _{1 \leq k \leq s}q_{k}(0)\gamma_k }
$$$$
\ddot{\beta}_t = \sum _{1 \leq k \leq s}q_{k}(t)\gamma_k  \  \ (\dagger)$$
where  $\{q_{k}(0),q_{k}(t)\}_{1 \leq k \leq s}\subseteq   \mathbb{Q}_{\geq0}$.

\begin{enumerate}
\item[] \textbf{Claim.} There is a set $\{\gamma_1^{\prime},\ldots,\gamma_s^{\prime}\} \in \prod_{i \in \mathbb{N}} \mathbb{Q} $ with $({\gamma_k^{\prime}})_i = 0$ for all $i > l$ of solutions of the following system of equations

$$ \dot{\alpha} = \sum _{1 \leq k \leq s}q_{k}(0){\gamma_k}^{\prime}$$$$ \dot{\beta_t} = \sum _{1 \leq k \leq s}q_{k}(t)\gamma_k^{\prime} \ \ (\ast)$$for all $ 1 \leq t \leq n$.

 Indeed, there are two possibilities.
First, suppose that $n+1 < s$. That is the number of equations is less than the number of indeterminates.
In this case $(\ast)$ has a solution.
Secondly, suppose that $n+1 = s$. Note that the matrix of coefficients of  $(\ast)$ and $(\dagger)$ are the same.
Since $(\dagger)$ has a solution, its matrix of coefficients is invertible. The same thing  holds for  $(\ast)$.
 Thus, in both cases, $(\ast)$ has a solution set $\{\gamma_1^{\prime},\ldots,\gamma_s^{\prime}\}$. Since  $(\dot{\alpha})_i =(\dot{\beta_k})_i=0$  for all $i > l$, we have $({\gamma_k^{\prime}})_i = 0$ for all $i > l$.
This yields the claim.
\end{enumerate}

 Now the set $\Gamma:=\{\gamma_1+\gamma_1^{\prime},\ldots, \gamma_s+\gamma_s^{\prime}\}$ is the set that we search for it.
\end{proof}

\begin{corollary}\label{maincor}
Let $M\subseteq\prod_{i \in \mathbb{N}} \mathbb{Q}$ be the set of all almost non-negative  and $I$-supported elements.
Then the semigroup  $\widetilde{H}:=M\cup\{0\}$   is a direct limit of $\{\mathbb{Q}_{\geq0}^{n_i}:i\in J\}$.
\end{corollary}

\begin{proof}
Look at $$\Gamma:=\{\sum_{i=1}^n \mathbb{Q}_{\geq0}\gamma_i:n\in\mathbb{N}| \{\gamma_i\}\subseteq\widetilde{H}\  \textit{ are } \textit{linearly independent over }\mathbb{Q} \}.$$
For each $\underline{\gamma}:=\sum_{i=1}^n \mathbb{Q}_{\geq0}\gamma_i\in\Gamma$, define $H_{\underline{\gamma}}$ by $\sum_{i=1}^n \mathbb{Q}_{\geq0}\gamma_i.$
Partially ordered $\Gamma$ by means of inclusion. Lemma \ref{mkey} implies that  $\Gamma$ is directed. Thus, $$\widetilde{H}\simeq\varinjlim_{\underline{\gamma}\in \Gamma} H_{\underline{\gamma}}.$$Let $\sum_{i=1}^n r_i\gamma_i\in H_{\underline{\gamma}}$, where $r_i\in\mathbb{Q}_{\geq0}$.
The assignment $\sum_{i=1}^n r_i\gamma_i\mapsto(r_1,\ldots,r_n)$ induces an isomorphism $\varphi_{\underline{\gamma}}:H_{\underline{\gamma}}\to\mathbb{Q}_{\geq0}^n$ of semigroups. Let $\underline{\gamma}\leq \underline{\gamma}'$
and $\rho_{\underline{\gamma},\underline{\gamma}'}:H_{\underline{\gamma}}\to H_{\underline{\gamma}'}$ be the natural inclusion.
Define $\psi_{n,n'}:\mathbb{Q}_{\geq0}^n\to \mathbb{Q}_{\geq0}^{n'}$ by $\varphi_{\underline{\gamma}'}\rho_{\underline{\gamma},\underline{\gamma}'}\varphi_{\underline{\gamma}}^{-1}$.
Then the direct limit of the direct system $\{\mathbb{Q}_{\geq0}^{\ell_n},\psi_{n,n'}\}$ is $\widetilde{H}$.
\end{proof}

\section{A toroidal direct system }

Our main result in this section is Theorem \ref {prod}.
We need several auxiliary lemmas.
We begin this section by recalling the following definition. Our references are \cite{BG}, \cite{BG1} and \cite{Gi}.
Also, \cite{BGT} contains many homological properties of semigroups.

\begin{definition}\label{deff}
Let $C\subseteq \mathbb{Z}^n$ be a semigroup and $k$ a field.
\begin{enumerate}
\item[(i)] Recall that  $k[C]$  is  the vector space $k^{(C)}$. Denote the basis element of $k[C]$ which corresponds to
$ c\in C$ by $ X^ c$. This monomial notation is suggested by the fact that $k[C]$ carries a natural multiplication whose table is given by $X^c X^{c^\prime}:= X^{c+c^\prime}$.
\item[(ii)] Recall that $C$ is\emph{ positive} if there is not any invertible element in $C$.
\item[(iii)] Recall that $C\subseteq \mathbb{Z}^n$ is called \emph{normal},
if whenever $c,c^\prime \in C$ and there is a positive integer $m$ such that $m( c-c^\prime) \in C$, then $ c-c^\prime \in C$.
\item[(iv)] Recall that a subsemigroup extension $C \subseteq \widetilde{C} $ is called \emph{full}, if whenever $h,h^\prime \in C$ and
$ h-h^\prime \in \widetilde{C}$ then  $h-h^\prime \in C$.
\item[(v)]Let $H \subseteq \mathbb{Q}^n$ be a $\mathbb{Q}_{\geq0}$-semigroup, i.e., a semigroup that is closed under scaler multiplication by $\mathbb{Q}_{\geq0}$. Recall that $H$ \emph{has no line}, if there is no nonzero vector in $H$ whose additive inverse is in $H$.
\end{enumerate}
\end{definition}

The following result plays an essential role in this paper.

\begin{lemma}\label{em}(see \cite{H1})
Let $V$ be a finite-dimensional vector space over $\mathbb{Q}$, $C\subseteq V$ is a finitely
generated $\mathbb{Q}_{\geq0}$-subsemigroup and $x\in V\setminus C$.\begin{enumerate}
\item[(i)]
Then there exists a linear
functional that is nonnegative on $C$ and negative on $x$.
\item[(ii)]If $C$ contains no line, one can choose $L$ so that it is positive on all nonzero elements of $C$.
\item[(iii)]
Set $C^\ast := \{f \in \Hom(V,\mathbb{Q})| f(C) \geq 0\}$.
Then $C^\ast$ is a finitely generated $\mathbb{Q}^+$-semigroup.
\end{enumerate}
\end{lemma}

\begin{example}\label{rem}
The finitely generated assumption of $C$ in Lemma \ref{em} is needed.
Indeed, look at $$H:=\{(a,b)\in\mathbb{ N}^2:a/b>1\}.$$ Then $H$
is normal. The cone it generated is $$C:=\mathbb{Q}_{\geq0}H=\{(a,b)\in\mathbb{Q}^2_{\geq0}:a/b>1\}.$$
Look at $x:=(2,2)$ and let $f:\mathbb{ Q}^2\to \mathbb{ Q}$ be any nonzero linear function that is nonnegative on $C$.
Note that $x\notin C$ and $f$ is continuous via the standard topology induced from $\mathbb{R}^2, \mathbb{R}^1$.
Define $a_n:=(2+1/n,2)$. Then $a_n\in C$ and $ \underset{n\to \infty}{\lim}a_n=x$. So
$$f(x)=\underset{n\to \infty}{\lim} f(a_n)\geq 0.$$
\end{example}

\begin{remark}\label{rem1}
Adopt the notation of Example \ref{rem}. One can prove that $C$ has not a minimal generating set as a $\mathbb{Q}_{\geq0}$-semigroup.
\end{remark}

We state the following result to demonstrate our interest on $\prod_{\mathbb{N}}\mathbb{Q}_{\geq0}$ and for possible application in the further.
\begin{remark}\label{hoc}
Let  $H\subseteq \mathbb{Z}^n$ be a positive and normal semigroup. There is an infinite index set $J$ such that
 $H$ is full in $\prod_{J} (\bigoplus_{\mathbb{N}}\mathbb{Q}\oplus\prod_{\mathbb{N}}\mathbb{Q}_{\geq0})$. Indeed, we assume that $H -H = \mathbb{Z}^{n_0}$. Define $C:=\mathbb{Q}_{\geq0}H\subseteq\mathbb{Q}^{n_0}$ as the  $\mathbb{Q}_{\geq0}$-subsemigroup generated
by $H$. Look at the vector spaces $V = \mathbb{Q}^{n_0}$ and $V^\ast=
\Hom_\mathbb{Q}(V, \mathbb{Q})$. Clearly, $C$ is a countable set.
Consider a chain $$C_1\subseteq C_2 \subseteq \ldots \subseteq C$$ such that
$C_i$ is  nonzero finitely generated  $\mathbb{Q}_{\geq0}$-subsemigroup of $C$ and
$\bigcup C_i=C$.
Now we define $\overline{C} \subseteq\prod_{\mathbb{N}} V^\ast$ by $$\overline{C}:=\{(f_n)_{n\in\mathbb{N}}|f_n\in V^\ast \textit{ and } f_n(C_n)\geq 0 \textit{ for all } n\}.$$
One may find that $\overline{C}$ is closed under sum and scaler multiplication by  $\mathbb{Q}_{\geq0}$.
Let $\{\textmd{a}^j:j\in J\}$ be the set $\overline{C}$ and denote the n-th component
of $\textmd{a}^j$ by $\textmd{a}^j_n$. Note that  $J$ is an infinite index set.
Fix $h\in H$ and $j\in J$. Then $h\in C$ and so $h\in C_n$ for some $n$.
Hence $h\in C_m$ for all $m\geq n$. This means that
$\textmd{a}^j(h)>0$ for all $m\geq n$. It turns out that
$$(\textmd{a}^j_n(h))_{n\in\mathbb{N}}\in(\bigoplus_{\mathbb{N}}\mathbb{Q}\oplus\prod_{\mathbb{N}}\mathbb{Q}_{\geq0}).$$
 So, the assignment $h \mapsto (\textmd{a}^j(h))_{j\in J}$, defines a map $$\varphi:H\longrightarrow \prod_{J} (\bigoplus_{\mathbb{N}}\mathbb{Q}\oplus\prod_{\mathbb{N}}\mathbb{Q}_{\geq0}).$$
We  show that
$\varphi$ is a full embedding.
To see this, let $x\neq y $ be two distinct  elements of  $H$.
Without loss of generality, we can assume
that $x-y\notin H$
 since $H$ is  positive.
Also, we can assume
that $x-y\notin C$
 since $H$ is normal.
 Then $x-y\notin C_n$ for all $n$.
In view of Lemma \ref{em}, there is a linear functional $f_n$ nonnegative
on $C_n$ and negative on $x-y$. Hence,
$j:=(f_n)_{n\in\mathbb{N}}\in \overline{C}$, i.e., we find
 $j$ and $n$ such that $\textmd{a}^j_n(x-y)< 0$. In particular,
$\textmd{a}^j_n(x)\neq \textmd{a}^j_n(y)$, i.e., $\varphi$ is injective.

We finish the proof by showing that  $\im (\varphi)\subseteq \prod (\bigoplus\mathbb{Q}\oplus\prod\mathbb{Q}_{\geq0})$ is  full.
Take $x,y\in H$ be such that $\varphi(x) -\varphi(y) \in \prod_{J} (\bigoplus_{\mathbb{N}}\mathbb{Q}\oplus\prod_{\mathbb{N}}\mathbb{Q}_{\geq0})$.
In order to show
$x-y\in H$, its enough to prove that  $x-y\in C$, because $H$ is normal. Suppose on the contrary that
$x-y\notin C$. Hence,  $x-y\notin C_n$ for all $n$. In the light of Lemma \ref{em},
there is a linear functional $f_n$
that is nonnegative on $C_n$ and negative on $x-y$. By definition of
$\overline{C}$, $j:=(f_n)_{n\in\mathbb{N}}\in \overline{C}$.
 Look at $j-$th component of $\varphi(x) -\varphi(y)$. It is $$(f_n(x-y))_{n\in\mathbb{N}}\in(\bigoplus_{\mathbb{N}}\mathbb{Q}\oplus\prod_{\mathbb{N}}\mathbb{Q}_{\geq0}).$$
So $f_i(x-y)\geq0$ for all but finitely many $i$, a contradiction that we search for it.
\end{remark}

\begin{lemma}\label{ha}
Let $C$ be a   normal submonoid  of  $\mathbb{Z}^n$. Then there is a direct system $\{(C_\gamma , f_{\gamma \delta })\}$ of finitely generated  normal submonoids  of $C$ such that  $C=\varinjlim_{\gamma \in \Gamma}C_\gamma$,  where $f_{\gamma \delta }: C_\gamma \to  C_\delta$ is the inclusion map for $ \gamma,\delta \in \Gamma$ with $ \gamma \leq \delta$.
\end{lemma}

\begin{proof}
This is in \cite[Lemma 2.2]{ADT}.
\end{proof}

Recall that a set $M\subseteq\prod_{\mathbb{N}}\mathbb{Q}$  is called \textit{almost positive}, if it consists of all $\beta \in \prod_{\mathbb{N}}\mathbb{Q}$ such that only finitely many coordinates of $\beta$ is negative.

\begin{lemma}\label{hoc}
Let  $H\subseteq \mathbb{Z}^n$ be a positive and normal semigroup
and let $\{h_1,\ldots,h_s\}$ be a finite subset of  $H$. The following holds.
\begin{enumerate}
\item[(i)]There is an almost positive $\mathbb{Q}^+$-subsemigroup $M\subseteq \prod_{\mathbb{N}}\mathbb{Q}$  and a
full embedding $\varphi:H\to M$.
\item[(ii)]There is an infinite set $I\subseteq\mathbb{N}$ such that $\varphi(h_k)_i>0$ for all $i \in I$ and $1 \leq k \leq s$.\end{enumerate}
\end{lemma}

\begin{proof}
Denote the group that $H$ generates by $H -H$ and suppose $H -H= \mathbb{Z}^{n_0}$. By Lemma \ref{ha}, $H = \bigcup_{i \in \mathbb{N}}H_i$ where $H_i$ is a finitely generated normal positive semigroup. Also, $H_i\subseteq H_{i+1}$. Without loss of the generality, we assume that $H_i -H_i = \mathbb{Z}^{n_0}$. Define $C:=\mathbb{Q}_{\geq0}H\subseteq\mathbb{Q}^{n_0}$ and $C_i:=\mathbb{Q}_{\geq0}H_i\subseteq\mathbb{Q}^{n_0}$  as the  $\mathbb{Q}_{\geq0}$-subsemigroups generated
by $H$ and $H_i$, respectively. Then $C_i$ is  a finitely generated $\mathbb{Q}^+$-semigroup with no line and   $C_i\subseteq C_{i+1}$.   Look at the vector space $V := \mathbb{Q}^{n_0}$ and its dual space $V^\ast:=
\Hom_\mathbb{Q}(V, \mathbb{Q})$. Set $C_i^\ast := \{f \in V^\ast| f(C_i) \geq 0\}$.
 In view of Lemma \ref{em}, $C_i^\ast$ is a finitely generated $\mathbb{Q}^+$-semigroup. Let $\{f_1^i,\ldots,f_{n_i}^i\}$ be a set of generators for $C_i^\ast$ as a $\mathbb{Q}^+$-semigroup. Also, take $g^i \in C_i^\ast$ such that $g^i(C_i) >0$. Such a linear functional exists by Lemma \ref{em}.
So, the assignment $$h \mapsto (g^1(h),f_1^1(h),\ldots,f_{n_1}^1(h);g^2(h),f_1^2(h),\ldots,f_{n_2}^2(h);\ldots),$$ defines a map $$\psi:H\longrightarrow \prod_{\mathbb{N}}\mathbb{Q}.$$

 We are ready to prove the Lemma.

(i): Denote the set of all $\beta \in \prod_{\mathbb{N}}\mathbb{Q}$ such that only finitely many coordinates of $\beta$ are negative by $M$.
 Clearly, $\psi(H) \subseteq M$. We now apply an idea from \cite{H1},
to show that the map
$\varphi:H\to M$  induced by $\psi$ is a full embedding.

Let $x\neq y $ be two distinct  elements of  $H$.
Without loss of the generality, we can assume
that $x-y\notin C$
 since $C$ has no any lines.
 Then $x-y\notin C_m$ for all $m$.
In view of Lemma \ref{em}, there is a linear functional $f_m$ nonnegative
on $C_m$ and negative on $x-y$. Hence, $f_m\notin C_i^\ast$. This means that
 $f_j^m(x-y)<0$ for some $1 \leq j \leq n_m$. Thus $\varphi(x) \neq \varphi(y)$ and so $\varphi:H\to M$ is an embedding.

 Let $x,y \in H$ be such that $\varphi(x)-\varphi(y) \in M$. Suppose on the contrary that $x-y \notin H$. Then for each $i$, there  exists
 $1 \leq j \leq n_i$ such that
  $f_j^i \in C_i^\ast$ and  $f_j^i(x-y) <0$. This implies that $\varphi(x)-\varphi(y) \notin M$, which is a contradiction. Therefore, in view of Definition \ref{deff}(iv), $\varphi(H)$ is full in $M$.

(ii): Let $i$ be such that $\{h_1,\ldots,h_s\}\subseteq C_i$. Keep in mind that $C_m\subset C_{m+1}$ for all $m$.
Then $\{g^j(h):j\geq i\}$ are components of $\varphi(h)$ that are nonzero for all $h\in\{h_1,\ldots,h_s\}$.
\end{proof}

\begin{lemma}\label{dir} Let $k$ be a field and $H$ be a full subsemigroup of  a semigroup $D$.  Then $k[H]$ is a direct summand of $k[D]$.
\end{lemma}

\begin{proof}
The proof is similar to the affine case and we leave it to the reader.
\end{proof}

\begin{definition}Let $R$ and $S$ be two polynomial rings over a filed and
$\varphi:R \to S$ a ring homomorphism. Then $\varphi$ is called a toric map if
it sends a monomial to a  monomial.
\end{definition}

\begin{theorem}\label{prod} Let $k$ be a field and $H\subseteq \mathbb{Z}^n$ be a positive and normal semigroup.  Then there is a direct system $\{A_n:n\in \mathbb{N}\}$  with the following properties:
\begin{enumerate}
\item[(i)] $A_n$ is a  Noetherian polynomial ring over $k$   for all  $n\in \mathbb{N}$.
\item[(ii)] $A_n \to A_m$ is toric for all $n\leq m$.
\item[(iii)] $k[H]$ is a direct summand of ${\varinjlim}_{n\in \mathbb{N}} A_n$.
\end{enumerate}
\end{theorem}

\begin{proof}
First, we remark that $H$ is countable. Thus, it has a countable generating set $\{h_i|i \in \mathbb{N}\}$. By Lemma \ref{hoc} (i), there is a full embedding $\varphi:H\hookrightarrow M\subset\prod_{\mathbb{N}}\mathbb{Q}$.
For any finite subset  $X$ of $H$, by applying Lemma \ref{hoc} (ii), there is an infinite set $I\subseteq\mathbb{N}$ such that $X$
consists of $I$-supported elements, when we regard $X$ as a subset of $M$.
So, we are in the situation of Lemma \ref{mkey}.

Fix $n\in \mathbb{N}$.
 In view of Lemma \ref{mkey}, there is a set $\Gamma:=\{\gamma_1^n,\ldots,\gamma_{s_n}^n\} \subseteq M$ of $\mathbb{Q}$-linearly independent elements and that $$\{h_1,\ldots,h_n\},\{\gamma_1^{n-1},\ldots,\gamma_{s_{n-1}}^{n-1}\} \subseteq \mathbb{Q}^+\gamma_1^n+ \ldots +\mathbb{Q}^+\gamma_{s_n}^n.$$ Since $\{h_1,\ldots,h_n\}, \{\gamma_1^{n-1},\ldots,\gamma_{s_{n-1}}^{n-1}\}$ are  finite, we can assume that $$\{h_1,\ldots,h_n\}, \{\gamma_1^{n-1},\ldots,\gamma_{s_{n-1}}^{n-1}\} \subseteq \mathbb{N}\gamma_1^n+ \ldots +\mathbb{N}\gamma_{s_n}^n.$$ Look at $C_n := \mathbb{N}\gamma_1^n+ \ldots +\mathbb{N}\gamma_{s_n}^n$. Hence $H \subseteq \bigcup_{n \in \mathbb{N}}C_n$. This is a full embedding, because
 $\bigcup_{n \in \mathbb{N}}C_n\subseteq M$ and $H \subseteq M$ is full.
Set $A_n:= k[C_n]$ and denote the natural map $A_n\to A_{n+1}$ by $\varphi_{n,n+1}$.

 Now we prove the Theorem.

 \begin{enumerate}
\item[(i)] $A_n$ is a  Noetherian polynomial rings over $k$, since $$C_n= \mathbb{N}\gamma_1^n+ \ldots +\mathbb{N}\gamma_{s_n}^n\simeq   \mathbb{N}^{s_n},$$  by the assignment $$m_1\gamma_1^n+ \ldots +m_{s_n}\gamma_{s_n}^n\longmapsto(m_1,\ldots,m_{s_n}).$$
\item[(ii)] $A_n \to A_{n+1}$ is toric, since $\varphi_{n,n+1}(C_n)\subset C_{n+1}$.
\item[(iii)] $k[H]$ is a direct summand of ${\varinjlim}_{n\in \mathbb{N}} A_n$, since $H \subseteq \bigcup_{n \in \mathbb{N}}C_n$ is a full embedding.
\end{enumerate}
  So,  $\{A_n:n\in \mathbb{N}\}$ is the desirable directed system.
\end{proof}

In the following  we cite a result of Teissier who studied a direct limit of a nested sequence of polynomial subalgebras
with toric maps  with applications on resolution of singularities.

\begin{remark}\label{t}Let $(R,\fm, k)$ be a valuation ring of finite Krull dimension containing $k$ with value map  $v:R\to\Gamma$. The associated graded ring of $R$ with respect to $v$ is
$$gr_v(R):=\bigoplus_{\gamma\in\Gamma} \{x\in R:v(x)\geq\gamma\}/\{x\in R:v(x)>\gamma\}.$$
In view of \cite[Section 4]{Te}, $gr_v(R)$  is a direct limit of a nested sequence of polynomial subalgebras with toric maps.
\end{remark}

\section{ Projective dimension and regular sequence }

In this section we present a criterion of regularity of sequences in the terms of projective dimension, see Theorem \ref{pd2}. We need it in  Theorem \ref{pro}.
 Our reference for combinatorial commutative algebra is \cite{HHI}.
Let $R$ be a ring, $M$ an $R$-module and  $\underline{x} = x_{1}, \ldots,
x_\ell$ be a system of elements of $R$. By $\mathbb{K}_{\bullet}(\underline{x})$,
we mean the Koszul complex of $R$ with respect to $\underline{x}$.
Also, $\pd_R(M)$ denotes the projective dimension of $M$ over $R$.

\begin{question}\label{q}
Let $\fa$ be an ideal of a ring $R$ minimally generated by $n$
elements. Suppose that $\pd(R/ \fa)=n$. Under what conditions
$\fa$  can be generated by a regular sequence?
\end{question}

There are several positive answers inspired by  \cite[Theorem 2.2.8]{BH}:

\begin{remark}
\begin{enumerate}
\item[(i)] Suppose $\fa$ is a parameter sequence of a local ring $R$ and $R$ contains a filed. By  using \textit{The Canonical Element Conjecture},
one can find a positive answer to  Question \ref{q}. For more details, see \cite[1.6.2, 1.6.3]{AI}.
\item[(ii)] Let $\underline{x}$ be a set of generators of $\fa$.
If $H^1(\mathbb{K}_{\bullet}(\underline{x}))$ is a free $R/\fa$-module, then one can find a positive answer to  Question \ref{q}. For more details, see \cite[Proposition 25]{S}.
\end{enumerate}
\end{remark}

\begin{lemma}\label{c}
Let  $A$  be a ring,  $x$ a regular element and $x_1,x_2$ a sequence  of elements of $A$.
If $xx_1,xx_2$ is a regular  sequence, then $x_1,x_2$ is a regular  sequence and
$\pd(A/(x_1,x_2)A)=\pd(A/(xx_1,xx_2)A)$.
\end{lemma}

\begin{proof}
Clear.
\end{proof}

\begin{lemma}\label{hilb}
Let  $A$  be a polynomial ring over a field and $x_1,x_2$ be a sequence of monomials in $A$.
If $\pd(A/(x_1,x_2)A)=2$, then  $x_1,x_2$
is a regular sequence in $A$.
\end{lemma}

\begin{proof}
 Look at the minimal free resolution of $A/(x_1,x_2)A$:$$0\longrightarrow P\stackrel{\varphi}\longrightarrow  A^2\longrightarrow A\longrightarrow A/(x_1,x_2)\longrightarrow 0\ \ (\ast).$$
Then $P$ is free, because finitely generated projective modules on
a polynomial ring over a field are free.
Localize $(\ast)$ with the fraction filed of $A$ to observe  $P=A$. In view of Hilbert-Burch Theorem \cite[Theorem 1.4.17]{BH}, we see that
$(x_1,x_2)=aI_1(\varphi)$ where $a$ is regular and $I_1(\varphi)$ is the first minor of $\varphi$. Its grade is two.
Rigidity of Koszul yields the claim. But we prefer to give a more directed proof.
By Lemma \ref{c}, we assume that $(x_1,x_2)=I_1(\varphi)$. Keep in mind that a monomial ideal  has a unique (monomial) minimal generating set by monomials (\cite[Proposition 1.1.6]{HHI}).
We use the proof of the converse part of  the Hilbert-Bruch Theorem \cite[Theorem 1.4.17]{BH} to obtain the following minimal free resolution of $ A/(x_1,x_2)$  $$0\longrightarrow A\stackrel{\binom{x_1}{x_2}}\longrightarrow  A^2\stackrel{\binom{-x_2}{x_1}^t}\longrightarrow A\longrightarrow 0\ \ (\ast,\ast).$$But $ (\ast,\ast)$ is the Koszul complex with respect to $x_1,x_2$. The acyclicity of the Koszul complex yields the claim.
\end{proof}

\begin{theorem}\label{pd2}
Let  $A$  be a Noetherian polynomial ring over a field  and $\underline{x}:=x_1,\ldots,x_n$ a monomial sequence in $A$.
If $\pd(A/(x_{i_1},\ldots, x_{i_k})A)=k$ for all $1\leq i_1<\ldots <i_k\leq n$, then  $\underline{x}$
is a regular sequence in $A$.
\end{theorem}

\begin{proof}
The proof is induction by on $n$. When $n=1$ the claim is clear. The case $n=2$ follows by Lemma \ref{hilb}.
Let $1\leq i\lneqq j\leq n$. By assumption,  $\pd(A/(x_{i}, x_{j})A)=2$. By  Lemma \ref{hilb}, $x_i,x_j$ is a regular sequence in $A$.
 Denote the least common multiple of $x_{i}, x_{j}$ by $[x_{i}, x_{j}]$. Thus
$$([x_{i}, x_{j}]/x_{i})=(x_{j}) \ \ (\star).$$By induction assumption,  $\underline{x}_{n-1}:=x_1,\ldots,x_{n-1}$ is a regular sequence. Then  $\mathbb{K}_{\bullet}(\underline{x}_{n-1})$
is a projective resolution of $A/\underline{x}_{n-1}A$. Also,  $\mathbb{K}_{\bullet}(x_n)$ provides
 a projective resolution for $A/x_nA$. This enable us to compute the following: $$H_i(\mathbb{K}_{\bullet}(\underline{x}))=H_i(\mathbb{K}_{\bullet}(\underline{x}_{n-1})
\otimes\mathbb{K}_{\bullet}(x_{n}))=\Tor_i^A(A/x_nA,A/\underline{x}_{n-1}A).$$ So $$H_i(\mathbb{K}_{\bullet}(\underline{x}))=0 \ \ \textit{for all } i>1$$and $$H_1(\mathbb{K}_{\bullet}(\underline{x}))=\Tor_1^A(A/x_nA,A/\underline{x}_{n-1}A)=\frac{(\underline{x}_{n-1})\cap (x_n)}{(\underline{x}_{n-1}A)(x_nA)}.$$ We apply \cite[Proposition 1.2.1]{HHI} to conclude that $$(\underline{x}_{n-1})\cap (x_n)=([x_{i}, x_{n}]:1\leq i\leq n-1).$$ In view of
$(\star)$, $(\underline{x}_{n-1})\cap (x_n)=(\underline{x}_{n-1}A)(x_nA)$. Therefore, $$\frac{(\underline{x}_{n-1})\cap (x_n)}{(\underline{x}_{n-1}A)(x_nA)}=0$$ and so $\mathbb{K}_{\bullet}(\underline{x})$ is acyclic. This means that
 $\underline{x}$ is a regular sequence, as claimed.
\end{proof}

The following says that the assumptions of Theorem \ref{pd2}
are needed.

\begin{example} (Engheta \cite{E}). Let $J$ be an ideal  of $R:=K[X_1,\ldots,X_n]$ generated by three cubic forms
and denote by $I$ the unmixed part of $J$.
\begin{enumerate}
\item[(i)] If $I$ contains a linear form, then $\pd(R/J)\leq3$.
\item[(ii)] The above bound is sharp.
\end{enumerate}
\end{example}
We use the following definition several times in the sequel.

\begin{definition}
Let $I$ be an ideal of a polynomial ring $S = K[y_1,\ldots , y_m]$ generated by
 monomials $\underline{x}:=x_1,\ldots,x_n$.  Let $T$ be a free S-module with basis
$\{e_1,\ldots , e_n\}$. Define $\mathbb{T}_i:= \wedge ^i T$ for $i = 0,\ldots , n$.
For each $\Delta = \{j_1 < \ldots < j_i\}$, the set $\{e_\Delta:\Delta\}$ provides
 a base for $\mathbb{T}_i$. Denote the least common multiple of the monomials
$\{x_i:i\in \Delta\}$ by $x_\Delta$. By $\sigma(\Delta,i)$ we mean the numbers of $j$
 with $j\in\Delta$ and $j<i$. Finally, define $\partial(e_\Delta):=\sum_{i\in\Delta}
(-1)^{\sigma(\Delta,i)}\frac{x_\Delta}{x_{\Delta\setminus\{i\}}} e_{\Delta\setminus \{i\}}$.
The \textit{Taylor complex} $\mathbb{T}$  with respect to $\underline{x}$ is  the following complex$$
\begin{CD}
0@>>>\mathbb{T}_0 @>\partial_{0}>>\cdots  @>>> \mathbb{T}_{n-2} @>\partial_{n-2}>> \mathbb{T}_{n-1} @>\partial_{n-1}>>\mathbb{T}_{n} @>>>0.\\
\end{CD}
$$
\end{definition}

\begin{remark}Here we give another proof for Theorem \ref{pd2}.
Denote the Taylor resolution of $\underline{x}$ by $\mathbb{T}(\underline{x})$.
First, we compute the $(n-1)$-th homology of $\mathbb{K}_{\bullet}(\underline{x})$. We do this by looking at $$
\begin{CD}
\mathbb{K}_{0} @>>>\cdots  @>>> \mathbb{K}_{n-2} @> d_{n-2}>> \mathbb{K}_{n-1}  @>d_{n-1}>>\mathbb{K}_{n} @>>>0\\
\varphi_{0}@VVV @VVV \varphi_{n-2}@VVV \varphi_{n-1}@VVV \varphi_{n}@VVV  \\
\mathbb{T}_0 @>>>\cdots  @>>> \mathbb{T}_{n-2} @>\partial_{n-2}>> \mathbb{T}_{n-1} @>\partial_{n-1}>>\mathbb{T}_{n} @>>>0.\\
\end{CD}
$$
The map $\varphi_{n}$ and $\varphi_{n-1}$ are identity. A diagram chasing argument shows that  $\varphi_{n-2}$ is a diagonal matrix
with terms $[x_{i}, x_{j}]/x_{i}x_{j}$ in the diagonal. In view of $(\star)$ in Theorem \ref{pd2}, $\varphi_{n-2}$ is an isomorphism.
Let $a\in \ker d_{n-1}$. Then $\varphi_{n-1}(a)\in \ker \partial_{n-1}=\im \partial_{n-2}$. There is an $b$ such that
$\varphi_{n-1}(a)= \partial_{n-2}(b)$. Take $c\in\mathbb{K}_{n-2}$ be such that $\varphi_{n-2}(c)=b$. Hence,
$d_{n-2}(c)-a\in \ker\varphi_{n-1}=0$. Thus $H^{n-1}(\mathbb{K}_{\bullet}(\underline{x})))=0$. The rigidity of the Koszul complex \cite[Corollary 1.6.9]{BH}, implies that $\underline{x}$ is a regular sequence.
\end{remark}

\section{ Invariant of tori\label{5.C}}

The main result of this section is Theorem \ref{mmain}.
We start our work in this section by recalling the concept
of parameter sequence
from \cite{HM}.

\begin{discussion}

Let $R$ be a ring, $M$ an $R$-module and  $\underline{x} = x_{1}, \ldots,
x_\ell$ a sequence of elements of $R$. For each $m \geq n$, there is a chain map $\varphi^{m} _{n} (\underline{x})
:\mathbb{K}_{\bullet}(\underline{x}^{m})\longrightarrow
\mathbb{K}_{\bullet}(\underline{x}^{n}),$ which induces via
multiplication by $(\prod x_{i})^{m-n}$.
Then $\underline{x}$ is called
weak proregular if for each $n>0$ there exists an $m \geq n$ such
that the maps $H_{i}(\varphi^{m} _{n} (\underline{x})) :H_{i}
(\mathbb{K}_{\bullet}(\underline{x}^{m}))\longrightarrow H_{i}
(\mathbb{K}_{\bullet}(\underline{x}^{n}))$ are zero for all $i \geq
1$.\end{discussion}

\begin{notation}By  $H^{\ell}_{\underline{x}}(M)$, we mean the $i$-th cohomology of the \v{C}ech complex of $M$ with respect to $\underline{x}$.
\end{notation}

\begin{definition}\label{defhm}(\cite[Definition 3.1]{HM})
Adopt the above notation. Then
$\underline{x}$ is called a \textit{parameter sequence} on
$R$, if:
(1) $\underline{x}$ is a weak proregular sequence; (2)
$(\underline{x})R \neq R$; and (3) $H^{\ell}_{\underline{x}}
(R)_{\fp} \neq 0$ for all $\fp \in
\V(\underline{x}R)$. Also, $\underline{x}$ is called a strong
parameter sequence on $R$ if $x_{1},\ldots, x_{i}$ is a parameter
sequence on $R$ for all $1\leq i \leq \ell$.
\end{definition}

\begin{notation}Let $\fa$ be an ideal of a ring $R$ with a generating set $\underline{x}$ and $M$ an $R$-module.
We denote
$\sup\{i\in \mathbb{Z}|H^i_{\underline{x}}(M)\neq 0\}$ by $\cd_{\fa}(M)$.
By $\cd(\fa)$ we mean $\sup\{\cd_{\fa}(M)|\textit{ M is  an R-module}\}$.
\end{notation}

\begin{lemma}\label{lyu}
Let $k$ be field and $I$ a monomial ideal in the polynomial ring $A:=k[X_1,\ldots,X_n]$. Then
 $\cd(I)\leq\pd(A/I)$.
\end{lemma}

\begin{proof} Let $t$ be an integer. The assignment $X_i\mapsto X_i^t$ induces a ring homomorphism
$F_t:k[X_1,\ldots,X_n]\to k[X_1,\ldots,X_n]$. By $F_t(A)$, we mean $A$ as a group equipped with  left and right scalar multiplication from $A$
given by
$$a.r\star b = aF_t(b)r, \  \ where \ \ a,b\in A  \ \ and \  \ r\in F_t(A),$$
Let $(F_\bullet,d_\bullet)$ be a finite free resolution of  $A/I$ with monomial maps. For example, take the Taylor resolution. Then $(F_\bullet,d_\bullet)\otimes_A F_t(A)=(F_\bullet,F_t(d_\bullet))$. Denote the ideal generated by the $\ell\times \ell$ minors of a matrix $(a_{ij})$ by $I_{\ell}((a_{ij}))$.
Let $r_i$ be the expected rank of $d_\bullet$.  For  more details and definitions, see \cite[Section 9.1]{BH}. Clearly,  $r_i$ is
the expected rank of $F_t(d_\bullet)$. Thus, $$\Kgrade_A(I_{r_i}(d_i), A)=\Kgrade_A(I_{r_i}(d_i^t ),A).$$
In view of \cite[Theorem 9.1.6]{BH},  $(F_\bullet,d_\bullet^t)$ is exact and so it is  a
free resolution of $F_t(A/I)\simeq A/(u^t:u\in I)$. It turns out that
$$\pd(A/I)=\pd(F_t(A/I))=\pd((A/(u^t:u\in I)).$$ Thus, $$\Ext^i_A(A/(u^t:u\in I), -)=0$$
for all $i>\pd(A/I)$. So,$$H^i_{I}(-)\simeq{\varinjlim}_{t\in \mathbb{N}}\Ext^i_A(A/(u^t:u\in I), -)=0,$$
which yields the claim.
\end{proof}

Also, we need:

\begin{lemma}\label{need}
Let  $A$  be a ring  and $\underline{x}:=x_1,\ldots,x_n$ a  sequence in $A$.
If $\cd(\underline{x}A)=n$, then $\cd((x_{i_1},\ldots, x_{i_k})A)=k$ for all $1\leq i_1<\ldots <i_k\leq n$.
\end{lemma}

\begin{proof}
Set $\underline{x}_{n-1}:=x_1,\ldots,x_{n-1}$ and look at the exact sequence$$H^{n-1}_{\underline{x}_{n-1}}(-)_{x_{n}}\longrightarrow
H^{n}_{\underline{x}}(-)\longrightarrow H^{n}_{\underline{x}_{n-1}}(-).$$Note that $H^{n}_{\underline{x}_{n-1}}(-)=0$ and
$H^{n}_{\underline{x}}(-)\neq0$. So, $H^{n-1}_{\underline{x}_{n-1}}(-)\neq 0$, i.e., $\cd(\underline{x}_{n-1}A_\delta)=n-1$. An easy induction yields the claim.
\end{proof}

\begin{theorem}\label{pro}
Let  $\{A_\gamma:\gamma\in \Gamma\}$  be a direct family of Noetherian polynomial rings over a field with toric maps and let $R$ be a direct summand of $A:= {\varinjlim}_{\gamma\in \Gamma} A_\gamma$.
Let $\underline{x}:=x_1,\ldots,x_\ell$ be a monomial strong parameter sequence in $R$. Then  $\underline{x}$
is a regular sequence in $R$.
\end{theorem}

\begin{proof}
By definition,
$H^{\ell}_{\underline{x}}
(R)_{\fp} \neq 0$ for all $\fp \in
\V(\underline{x}R)$, and so $H^{\ell}_{\underline{x}}
(R) \neq 0$. There is an $R$-module $M$ such that
$R\oplus M=A$. Thus, $H^{\ell}_{\underline{x}}
(R) \oplus H^{\ell}_{\underline{x}}
(M)\simeq H^{\ell}_{\underline{x}}
(A)$. Hence, $H^{\ell}_{\underline{x}}
(A)\neq0$. Keep in mind that $A= {\varinjlim}_{\gamma\in \Gamma} A_\gamma$. It yields that
 there is a direct set $\Lambda$ cofinal with respect to $\Gamma$ such that
 $H^{\ell}_{\underline{x}}
(A_\delta)\neq0$ for all $\delta\in\Delta$. Thus,
$\cd(\underline{x}A_\delta)=\ell$.
Clearly,
$\underline{x}$ is  monomial in $A_\delta$.
Also, the Taylor resolution provides a bound for the projective resolutions  of monomial ideals. In the light of Lemma
\ref{lyu},  \[\begin{array}{ll}
\ell&=\cd(\underline{x}A_\delta)\\
&\leq \pd(A_\delta/\underline{x}A_\delta)\\&\leq\ell.
\end{array}\]
Therefore,
$$\pd(A_\delta/\underline{x}A_\delta)=\ell  \ \ (\ast).$$
 By Lemma \ref{need},  $\cd(x_{i_1},\ldots, x_{i_k})=k$ for all $1\leq i_1<\ldots <i_k\leq \ell$.
In view of $(\ast)$  we see
$$\pd(A_\delta/(x_{i_1},\ldots, x_{i_k}))=k$$for all $1\leq i_1<\ldots <i_k\leq \ell$.
Due to Theorem \ref{pd2},
 $\underline{x}$ is a regular sequence in $A_\delta$ and so
in $A:= {\varinjlim}_{\delta\in \Delta} A_\delta$. Since $R$ is pure in $A$, $\underline{x}$ is a regular sequence in $A$ by \cite[Lemma 6.4.4(c)]{BH} and this completes the proof.
\end{proof}

\begin{remark}
Adopt the notation of
Theorem \ref{pro}. One can proof it
by using polarization instead of the Taylor resolution.
To see this, let $(\underline{y})\subseteq B_\delta:=A_\delta[Y_1,\ldots,Y_{n_\delta}]$
be a polarization of $(\underline{x})A_\delta\subseteq A_\delta$ and remark  by \cite{Lyu} in the  square-free case  that  the inequality of Lemma
\ref{lyu} achieved. So  \[\begin{array}{ll}
\ell&\leq \pd(A_\delta/\underline{x}A_\delta)\\
&= \pd(B_\delta/(\underline{y})B_\delta)\\
&=\cd(\underline{y}B_\delta) \\
&\leq\mu(\underline{y}B_\delta)\\
&=\sum_j\beta_{1j}(\underline{y}B_\delta)\\
&=\sum_j\beta_{1j}(\underline{x}A_\delta)\\&\leq\ell.
\end{array}\]
\end{remark}

We need the following:

\begin{lemma} \label{par} Let $R$ be a ring and $\underline{x}:=x_1,\ldots,x_{\ell}$ a finite sequence of elements of $R$.
\begin{enumerate}
\item[(i)] Let $f :R\to S$ be a flat ring homomorphism. If $\underline{x}$ is a (strong) parameter sequence on R and
$S/(f (\underline{x}))S \neq 0$ then $f (\underline{x})$ is a (strong) parameter sequence on $S$. The converse holds if $f$ is
faithfully flat.
\item[(ii)] If $\underline{y}:=y_1,\ldots,y_{\ell}$ is such that
$\rad(\underline{y})R =\rad(\underline{x})R$, then $\underline{x}$ is a parameter sequence on R if and only if $\underline{y}$ is
a parameter sequence on $R$.
\item[(iii)] If $u_1,\ldots,u_{\ell}$ are invertible and
$\underline{y}:=x_1u_1,\ldots,x_{\ell}u_{\ell}$, then $\underline{x}$ is a regular sequence if and only if $\underline{y}$ is a regular sequence.
\end{enumerate}
\end{lemma}
\begin{proof}
Parts (i) and (ii) are in \cite[Lemma 3.3]{HM}.
Part (iii) is easy and we leave it to the reader.
\end{proof}

\begin{lemma}\label{lo}
Let $A$ be a $\mathbb{Z}^n$-graded ring such that any monomial parameter sequence of
$A$ is regular. Then any monomial parameter sequence of
$A[X_1,\ldots, X_k,X_{1} ^{-1},\ldots, X_{k} ^{-1}]$ is regular.
\end{lemma}

\begin{proof}
Denote  $A[X_1,\ldots, X_k,X_{1} ^{-1},\ldots, X_{k} ^{-1}]$ by $A^e$.  Let $\textbf{\underline{x}}:=\textbf{u}_1,\ldots,\textbf{u}_{\ell} $ be a monomial parameter sequence on $A^e$.
Take $\textbf{v}_j\in A$ be such that  $\textbf{u}_j=\textbf{v}_j\textbf{w}_j$, where $\textbf{w}_j$ is a monomial in terms of
$X_i$ and $X_i^{-1}$ and look at $\textbf{\underline{y}}:=\textbf{v}_1,\ldots,\textbf{v}_{\ell}$. In view of Lemma \ref{par}, $\textbf{\underline{y}}$ is a monomial parameter sequence in $A^e$, since $\rad(\underline{y})A^e =\rad(\underline{x})A^e$. Keep in mind that $A\to A^e$ is faithfully flat. Again, by Lemma \ref{par}, $\textbf{\underline{y}}$ is a monomial parameter sequence in $A$. Thus,
$\textbf{y}$ is a regular sequence in $A$. It turns out that $\textbf{y}$ is  a regular sequence in $A^e$, because $A\to A^e$ is faithfully flat.
In view of Lemma \ref{par}, $\textbf{x}$ is a regular sequence in $A^e$.
\end{proof}

\begin{lemma}\label{mmmain} Let $k$ be a field and $H\subseteq \mathbb{Z}^n$ a positive and normal semigroup. Then
any monomial parameter sequence of
$k[H]$ is regular.
\end{lemma}

\begin{proof}
In view of Theorem \ref {prod}, there is a direct system $\{A_\gamma:\gamma\in \Gamma\}$  of Noetherian regular domains containing $k$ such that $k[H]$ is a direct summand of ${\varinjlim}_{\gamma\in \Gamma} A_\gamma$. By construction, $A_\gamma\to A_\delta$ is toric for all $\gamma\leq\delta$.  The claim follows by  Theorem \ref {pro} and Lemma \ref{lo}.
\end{proof}

 The proof of the next result is exactly  similar to the affine case. For the convenience of the reader we give its proof.

\begin{lemma}\label{ho}
Let $C$ be a normal subsemigroup of $\mathbb{Z}^n$. Then $ C\cong \mathbb{Z}^k\oplus C^\prime$, where $C^\prime$ is isomorphic to a positive normal semigroup.
\end{lemma}

\begin{proof}
    Without loss of the generality, we may assume that $C-C = \mathbb{Z}^n$. Let $H$ be the set of all elements of $C$ with additive inverse in $C$. Then $H$ is a subgroup of $\mathbb{Z}^n$, and so $H \cong \mathbb{Z}^k$ for some $k \in\mathbb{N}$.
     Suppose $\beta \in \mathbb{Z}^n = C-C$ and $\ell\beta \in H$. Then $\ell(-\beta)\in H$ as well. Both $\beta$ and $-\beta$ are in $\mathbb{Z}^n = C-C $. It follows that $\beta$ and $-\beta$ are both in $C$. So $\beta \in H$, as required. Thus, $\mathbb{Z}^n /H$ is a finitely generated torsion-free group.  Conclude that it is free. Thus,
$$0\longrightarrow H \longrightarrow \mathbb{Z}^n \longrightarrow \mathbb{Z}^n /H \longrightarrow 0$$
splits. Let $H^\prime$ be a free complement for $H$ in $\mathbb{Z}^n$. Every element $\beta \in C$ can be expressed uniquely as $\alpha +\alpha ^\prime$ where $ \alpha \in H$ and $ \alpha ^\prime \in  H^\prime$. But $-\alpha \in C$, and so $\alpha ^\prime \in C$. Thus $C = H \oplus  C^\prime$, where $ C^ \prime = C \cap H^ \prime$.  An easy computation shows that $C^ \prime$ is positive and normal.
\end{proof}

\begin{theorem}\label{mmain} Let $k$ be a field and $H\subseteq \mathbb{Z}^n$ be a normal semigroup. Then
any monomial parameter sequence of
$k[H]$ is a regular sequence.
\end{theorem}

\begin{proof}
In view of Lemma \ref{ho},  $H\cong \mathbb{Z}^k\oplus H^\prime$, where $H^\prime$ is isomorphic to a positive normal semigroup.
It is easy to see that $$k[H]\cong k[\mathbb{Z}^k\oplus H^\prime]\cong k[H^\prime][X_1,\ldots, X_k,X_{1} ^{-1},\ldots, X_{k} ^{-1}].$$
By Lemma \ref{mmmain}, any monomial parameter sequence  is  a regular sequence.  We use Lemma \ref{lo} to deduce the claim.
\end{proof}

\section{All together now: The Proof of Theorem 1.1}

We start this section by the following.

\begin{notation}
Denote $\mathbb{N}^\infty := \bigcup_{s\in \mathbb{N}} \mathbb{N}^s$ and  $\bigcup_{s\in \mathbb{N}} \mathbb{Z}^s$  by $\mathbb{Z}^\infty$.
\end{notation}

\begin{definition}\label{ful}
Let  $ H\subseteq \mathbb{Z}^\infty$ be a  semigroup.
\begin{enumerate}
\item[(i)] Suppose  $\alpha, \alpha^\prime \in H$ and there is $k \in \mathbb{N}$ such that $k(\alpha -\alpha^\prime) \in H$.
We say  $H$ is normal,  if $\alpha -\alpha^\prime \in H$.
\item[(ii)] Suppose $\alpha, \alpha^\prime \in H$ and $\alpha -\alpha^\prime \in  \mathbb{N}^\infty $. We say $H$ is full,
if $\alpha -\alpha^\prime \in H$.\end{enumerate}
\end{definition}

\begin{definition}\label{cmhm}
A ring is called Cohen-Macaulay in the sense of Hamilton-Marley, if any  of its
strong parameter sequence is a regular sequence.
\end{definition}

\begin{lemma}\label{pdhm}
Let $\{A_\gamma:\gamma\in \Gamma\}$ be a direct family of Cohen-Macaulay rings in the sense of Hamilton-Marley.
If $A_\delta\to A_\gamma$ is pure for all $\delta\leq\gamma$, then $A:={\varinjlim}_{\gamma\in \Gamma} A_\gamma$ is
a Cohen-Macaulay ring in the sense of Hamilton-Marley.
\end{lemma}

\begin{proof}
Let $\textbf{x}:=x_1,\ldots,x_n $ be a strong parameter sequence on $A$. Take $\gamma \in \Gamma$ be such that $\textbf{x}\in A_\delta $ for all $\delta \geq \gamma$. There exists an
integer $m \geq n$ such that the maps
$$\varphi_{m,n,A}:=H_{i}(\varphi^{m} _{n} (\underline{x};A)) :H_{i}
(\mathbb{K}_{\bullet}(\underline{x}^{m};A))\longrightarrow H_{i}
(\mathbb{K}_{\bullet}(\underline{x}^{n};A))$$ are zero for all $i
\geq 1$. By purity and in view of \cite[Ex. 10.3.31]{BH}, there is the following
commutative diagram

$$
\begin{CD}
0\longrightarrow H_{i}
(\mathbb{K}_{\bullet}(\underline{x}^{m};A_\delta)) @>>> H_{i}
(\mathbb{K}_{\bullet}(\underline{x}^{m};A))\\
\varphi_{m,n,A_\delta}
@VVV  \varphi_{m,n,A}@VVV  \\
0\longrightarrow H_{i}
(\mathbb{K}_{\bullet}(\underline{x}^{n};A_\delta)) @>>> H_{i}
(\mathbb{K}_{\bullet}(\underline{x}^{n};A))\\
\end{CD}
$$
with exact rows. Thus, $$H_{i}(\varphi^{m} _{n} (\underline{x};A_\delta)) :H_{i}
(\mathbb{K}_{\bullet}(\underline{x}^{m};A_\delta))\longrightarrow H_{i}
(\mathbb{K}_{\bullet}(\underline{x}^{n};A_\delta))$$ are zero for all $i
\geq 1$, i.e.,  $\textbf{x}$ is a weak proregular sequence on  $A_\delta $.
Let $\fp \in \Var(\textbf{x}A_\delta )$. There is $\fq \in \Spec (A)$ such that
$\fq \cap A_\delta= \fp $, because  $A_\delta  \hookrightarrow A$ is pure (note that the lying over property  is true for pure morphisms). Then,
\[\begin{array}{ll}
H_{\textbf{x}}^{n}(A)_{\fq}&\cong H_{\textbf{x}A}^{n}(A_\fq)\\
&\cong H_{\textbf{x}(A_\delta)_\fp}^{n}(A_\fq)\\
&\cong H_{\textbf{x}}^{n}((A_\delta) _\fp)\otimes_{(A_\delta )_\fp} A_{\fq}. \\
\end{array}\]
It yields that $H_{\textbf{x}}^{n}((A_\delta)_\fp) \neq 0$. Hence $\textbf{x}$ is a strong parameter sequence on $A_\delta$. So $\textbf{x}$ is a regular sequence on $A_\delta$. Therefore, $\textbf{x}$ is a regular sequence on $A= {\varinjlim}_{\gamma\in \Gamma} A_\gamma$.
\end{proof}

\begin{lemma}\label{wpdhm}
Let $\{A_\gamma:\gamma\in \Gamma\}$ be a direct family of Cohen-Macaulay graded rings with pure morphisms such that
their monomial parameter sequences are regular sequences. Then any monomial parameter sequence of $A:= {\varinjlim}_{\gamma\in \Gamma} A_\gamma$ is
a regular sequence.
\end{lemma}

\begin{proof}
The proof is similar to the proof of Lemma \ref{pdhm}.
\end{proof}

The preparation of Theorem 1.1 in the introduction is finished. Now,
we proceed to the proof of it. We repeat Theorem 1.1 to give its
proof.

\begin{theorem}\label{main} Let $k$ be a field and $H\subseteq \mathbb{Z}^\infty$ be a normal semigroup. Then
any monomial parameter sequence of
$k[H]$ is a regular sequence.
\end{theorem}

\begin{proof}
For each $n$, define $H(n):=H\cap \mathbb{Z}^n$.
In view of Definition \ref{ful}, $H(n)\subseteq H(n+1)$ is full. This yields that $k[H(n)]\subseteq k[H(n+1)]$
is pure. Clearly, $H(n)\subseteq \mathbb{Z}^n$ is normal.
By Theorem  \ref{mmain}, any monomial parameter sequence of $K[H(n)]$ is a regular sequence.
Thus, we are in the situation of Lemma \ref{wpdhm}, and so
any monomial parameter sequence of
$k[H]$ is a regular sequence.
\end{proof}

\begin{remark} In the proof of Theorem \ref{main} we use
only the properties 1) and 3) of Definition \ref{defhm}
but not 2).
\end{remark}

\section{An example in practise: quasi rational plane cones}

One source of producing  2-dimensional  Cohen-Macaulay rings is the Serre's characterization of
normality in terms of his conditions $(S_2)$ and $(R_1)$. Let $C \subseteq \mathbb{Z}^2$ be a normal semigroup.
It implies that $k[C]$ is Cohen-Macaulay.
Note that Serre's characterization of
normality is a result about noetherian rings. In fact there are 2-dimensional non-noetherian  normal integral domains
that they are not Cohen-Macaulay in a sense. We can take such rings that come from a normal semigroup $C \subseteq \mathbb{Z}^2$.

$$\textmd{Subsection 8.1: Convenience}$$

\begin{discussion}\label{disdef}
Let $f:\mathbb{R}^2\to \mathbb{R}$ be a linear form. The open half space associated to $f$
defined by $$H^>_f:=\{x\in \mathbb{R}^2:f(x)> 0\}.$$ The closed half space associated to $f$
defined by $$H^{\geq}_f:=\{x\in \mathbb{R}^2:f(x)\geq 0\}.$$

Let $L_1$ and $L_2$ be two half spaces define by the linear forms $l_1$ and $l_2$ with  rational slopes.
Half spaces are not  necessarily closed.
By a \textit{quasi-rational plane cone}, we mean $L_1\cap L_2$.
We assume that  $L_1\cap L_2$ is positive.  Through this section
$D$ is the lattice point of a quasi-rational plane cone.
We are interested  on semigroups $C\subseteq D$ such that the extension is  full and integral.
The reason of this interest is because of the Subsection 8.5.  
\end{discussion}

\begin{notation} Let  $l_1$ and $l_2$  be two half-lines in  the plane cross to origin.
 We denote  the convex section that $l_1$ and $l_2$ generate by $\conv(l_1,l_2)$.
 We denote the anti-clock  angel from $l_2$ to $l_1$ by $\angle(l_2,l_1)$.
\end{notation}

\begin{proposition}\label{ch}
Let $C$ be a normal submonoid of $\mathbb{Z}^2$ defined by Discussion \ref{disdef} and suppose that  $C$ is not finitely generated and  $C-C = \mathbb{Z}^2$. Then $C$ is  isomorph to one of the following semigroups.
\begin{enumerate}
\item[(i)]$ H := \{(a,b) \in \mathbb{N}_0^2| 0 \leq b/a < \infty\} \cup \{(0,0)\}$ or a full semigroup $M$ of $H$ such that for each $(a,b) \in H$, there exists $k \in \mathbb{N}$ such that $k(a,b) \in M$.
\item[(ii)] $ H^{\prime} := \{(a,b) \in \mathbb{N}^2| 0 < b/a < \infty\} \cup \{(0,0)\}$ or a full semigroup $N$ of $H^{\prime}$ such that for each $(a,b) \in H^{\prime}$, there exists $k \in \mathbb{N}$ such that $k(a,b) \in N$.
\item[(iii)]$ H_1 := \{(a,b) \in \mathbb{Z}^2|  b \in \mathbb{N}_0,\; if \;a \; is \; negative\; b\neq 0\} \cup \{(0,0)\}$ or a full semigroup $M$ of $H_i$ such that for each $(a,b) \in H_1$, there exists $k \in \mathbb{N}$ such that $k(a,b) \in M$.
    \item[(iv)]$ H_2 := \{(a,b) \in \mathbb{Z}^2|  b \in \mathbb{N}\} \cup \{(0,0)\}$, or a full semigroup $M$ of $H_2$ such that for each $(a,b) \in H_2$, there exists $k \in \mathbb{N}$ such that $k(a,b) \in M$.
\end{enumerate}
\end{proposition}

\begin{proof}
 First suppose that $\pi/2<\angle(l_2,l_1)<\pi$.
 Set $M:=\conv(l_1,l_2)\cap\mathbb{Z}^2$.  Then by
\cite[Corollary 2.10]{BG}, $M$ is finitely generated.
Choose $p_i$ in $M$ with the property that they are first integer point of $l_i$ from the origin.
Note that $M\subset\mathbb{Q}^+p_1 + \mathbb{Q}^+p_2$. So there is an $n \in \mathbb{N}$ such that   for each point $P$ of $M$, $nP \in \mathbb{N}p_1+ \mathbb{N}p_2$. Define the linear map $$\psi: \mathbb{Q}^2 \to \mathbb{Q}^2$$via the assignments
$$p_1\mapsto n(1,0) \  \ \textit{ and }  \  \  p_2\mapsto n(0,1).$$ Then $\psi$ is an isomorphism. For each $m \in M$, write $m = q_1p_1+q_2p_2$ with $nq_1,nq_2 \in \mathbb{N}$. Hence $\psi(nm)=n^2q_1(1,0)+q_2n^2(0,1)$.
Conclude that  $\psi(M) \subseteq \mathbb{N}^2$. Furthermore, at least one of the axes dose'nt intersect with $\psi(C)$. Without loss of the generality, we may assume that this axis is the $y$-axis. Thus we are in the situation of  $(i)$ and $(ii)$

Secondly, suppose that $0<\angle(l_2,l_1)<\pi/2$. In this case, similar as the first case, we achieve the above items $(i)$ and $(ii)$.

Thirdly, suppose that $\angle(l_2,l_1)$  is $\pi$  radian. In this case, similar as above,  there exists a linear assignment $\eta$ such that  $\conv(l_1,l_2)$ maps isomorphically  to $W:= \{(\alpha , \beta ) \in \mathbb{Z}^2| \beta \geq 0\}$.
Thus we are in the situation of  $(iii)$ and $(iv)$.
\end{proof}

$$\textmd{Subsection 8.2.  Preliminary lemmas}$$

 We start with the following.

\begin{remark} \label{adt}Let $H$ be the normal semigroup $  \{(a,b) \in \mathbb{N}_0^2| 0 \leq b/a < \infty\} \cup \{(0,0)\}$ and let $k$ be a  field.
Note that $H$ is normal.
In view of \cite{ADT}\begin{enumerate}
\item[$\mathrm{(i)}$]$k[H]$ is not  Cohen-Macaulay in the sense of ideals. This means that there is an
ideal  $\fa$  such that $\Ht(\fa)\neq\pgrade(\fa, k[H])$, where  $\pgrade(\fa, k[H])$ is   the polynomial grade of $\fa$.
\item[$\mathrm{(ii)}$]$k[H]$ is not weak Bourbaki unmixed. This means that there is a  finitely generated
ideal  $\fa$ of height greater or equal than the minimal number of its generator  such that minimal prime ideals $\fa$  does not coincide  with the set of all weak associated prime ideals of $k[H]/\fa$.
\item[$\mathrm{(iii)}$]
 $k[H]$ is  Cohen-Macaulay in the sense of Hamilton-Marley.
\end{enumerate}
Also,   Cohen-Macaulayness is not closed under taking the direct limit, if we adopt
each of the above notion as a candida for definition of non-Noetherian Cohen-Macaulay rings. For more details see \cite{ADT}.
\end{remark}

We use the following several times in this paper.

\begin{lemma}(see \cite[Theorems 21.4 and 17.1]{Gi})\label{g}
Let $H\subseteq \mathbb{Z}^n$ be a semigroup with $\mathbb{Z}^n$ as a group that it generates. Then $\dim k[H]=\dim k[\mathbb{Z}^n]=n$.
\end{lemma}

\begin{lemma}\label{2}
Let $H$ be as Remark \ref{adt} and let
$f\in k[H]$ be such that $f(0)\neq 0$.  If $\fp \in \Var_{k[x,y]}(f)$ is of height one, then there is $f_1\in k[H]$ such that $\fp=f_1k[x,y]$ and $f_1(0)\neq 0$. Also, $fk[H] = (f) k[x,y]  \cap k[H]$.
\end{lemma}

\begin{proof}
This is in \cite[Lemma 4.9]{ADT}  and the proof of  \cite[Theorem 4.10]{ADT}.
\end{proof}

\begin{lemma}\label{2.1}
Let $\underline{x}:=x_1,\ldots,x_n$ be a parameter sequence. Then $\Ht(\underline{x})\geq n$.
\end{lemma}

\begin{proof}
See \cite[Proposition 3.6]{HM}.
\end{proof}

$$\textmd{Subsection 8.3:  Certain Cohen-Macaulay rings}$$

\begin{notation}
Let $f\in k[H]$ and $h\in H$. By $f_h$ we mean the coefficient of $ X^ h$ in $f$.
\end{notation}
\begin{discussion}
In this subsection we show all semigroups appear in Proposition \ref{ch} are Cohen-Macaulay.
The proofs have the same sprit, but different details. The strategy is as follows.
Take $f,g$ be a parameter sequence. Combining Lemmas \ref{2.1} and \ref{g} we have $\Ht(f,g)k[C] = 2$. Without loss of the generality, we assume that
$f_{(0,0)}\neq 0$. Then we construct a ring homomorphism from our affine toric ring $A$ to a ring
$B$ with the property  $fA = (f) B  \cap A$ and
$f,g$ is a regular sequence in $B$. Then our proofs become complete, by showing that
the same thing holds in $A$.

To find a contradiction of certain contrary, one of our tricks is to find an element $h$
such that $\Var(f,g)=\Var(h)$.
\end{discussion}

  \begin{lemma}
  Let $C \subseteq \mathbb{Z}^2$ be a semigroup which is isomorph to $ H := \{(a,b) \in \mathbb{N}_0^2| 0 \leq b/a < \infty\} \cup \{(0,0)\}$ or a full semigroup $M$ of $H$ such that for each $(a,b) \in H$, there exists $t \in \mathbb{N}$ such that $t(a,b) \in M$. Then $k[C]$ is Cohen-Macaulay in the sense of Hamilton-Marley.
  \end{lemma}

  \begin{proof}
  If $C$ is isomorph to $H$ then by Remark \ref{adt}, $k[C]$ is Cohen-Macaulay. Thus $C$ is isomorph to $M$.
  Note that $\fm := \{ h \in k[C]|h(0,0)=0\}$ is a maximal ideal of $k[C]$.
    Take  $f,g \in \fm$.  There exists $k \in \mathbb{N}$ such that $x^k \in k[C]$.  If $x^l y^d \in k[C]$, then there exists $t\in \mathbb{N}$ such that $(2lkt-kt,2dkt) \in C$. So $$(x^ly^d)^{2kt} = x^{kt}(x^{2lkt-kt}y^{2kdt}) \in \fp$$ for all $\fp \in \Var(x^k)$.
   This implies that $\Var(f,g)=\Var(x^k)$. In view of \cite[Proposition 2.1(e)]{HM} $$H_{f,g}^2(k[C])_{\fm} = H_{x^k}^2(k[C])_{\fm} =0.$$ Then $f,g$ isn't a parameter sequence.

  Now we assume that $f,g$ is a parameter sequence such that $f(0,0) \neq 0$. Since $C$ is full in $H$, $k[C]$ is direct summand of $k[H]$. Furthermore, for each $x^ly^d \in k[H]$, there is $t \in \mathbb{N}$ such that $(x^ly^d)^t \in k[C]$. Therefore $k[H]$ is integral over $k[C]$. Note that  $k[C]$ is normal. This implies that the inclusion map $k[C] \to k[H]$ has the going down and going up property. Since $f,g$ is a parameter sequence on $k[C]$, combining Lemmas \ref{2.1} and \ref{g} we have $\Ht(f,g)k[C] = 2$.
 We claim that $\Ht(f,g)k[x,y] =2$. Else, by Lemma \ref{2}, there is $f_1 \in k[H]$ such that $f_1k[x,y] \in \Var_{k[x,y]}(f,g)$. Suppose $f_1 = c + xh$ where $0 \neq c \in k , h \in k[x,y]$.  Thus $(f_1,hxy)k[x,y]$ is proper in $k[x,y]$. Suppose on the contrary that
  there are $f^{\prime},g^{\prime} \in k[x,y]$ such that $ f^{\prime}(c+xh)+g^{\prime}(hxy) =1$. Then $f^{\prime}= 1/c$ and $1/c x + g^{\prime}x y =0 $, that's impossible. Hence $$f_1 k[H] \subseteq (f_1,hxy)k[x,y] \cap k[H]  \subsetneqq k[H].$$ Besides $hxy \notin f_1 k[H]$, $\Ht(f_1k[H])=1$. Thus $\Ht(f_1k[H] \cap k[C]) =1 $ and $f,g \in \Ht(f_1k[H] \cap k[C])$. This  is a contradiction. So $\Ht(f,g)k[x,y] =2$. In particular, $f,g$ is a regular sequence in $k[x,y]$. By Lemma \ref{2}, $fk[H] = fk[x,y] \cap A$. Therefore $f,g$ is a regular sequence in $k[H]$. By purity of the inclusion map $k[C] \to k[H]$, $f,g$ is a regular sequence in $k[C]$. This finishes the proof.

  \end{proof}

  \begin{lemma}\label{H'}
  Let $ H^{\prime} := \{(a,b) \in \mathbb{N}^2| 0 < b/a < \infty\} \cup \{(0,0)\}$. Then $k[H^{\prime}]$ is Cohen-Macaulay in the sense of Hamilton-Marley.
  \end{lemma}

  \begin{proof}
 Note that $\fm := \{ h \in k[H^{\prime}]|h(0,0)=0\} $ is a maximal ideal of $k[H^{\prime}]$ and take $f,g \in \fm$. An easy computation implies that $\Var(xy)= \fm$. Therefore $$H_{f,g}^2(k[H^{\prime}])_{\fm} = H_{xy}^2(k[H^{\prime}])_{\fm} =0.$$ So $f,g$ can't be a parameter sequence.

  Now we assume that $f,g$ is a parameter sequence in $k[H^{\prime}]$ such that $f(0,0) \neq 0$. Since $f,g$ is a parameter sequence on $k[H^{\prime}]$, we have $\Ht(f,g)k[H^{\prime}] = 2$. We bring the following:

\textbf{Claim.}  $\Ht(f,g)k[x,y] =2$.

Indeed,
else, by Lemma \ref{2}, there is $f_1 \in k[H]$ such that $f_1k[x,y] \in \Var_{k[x,y]}(f,g)$.  Let $k[\hat{H}] := k + yk[x,y]$.
 The assignment $m+xn\mapsto m+yn$ gives us an isomorphism $\phi:k[H]\to k[\hat{H}]$. Apply this to the reasoning of Lemma \ref{2}, we have\\
 $$f_1 \in k[\hat{H}], \ \ and \ \ f_1k[x,y] \cap k[\hat{H}] = f_1k[\hat{H}].  \ \ (\dag)$$
 So $f_1 \in k[H^{\prime}]$. Again, by Lemma \ref{2}, $f_1k[x,y] \cap k[H] = f_1k[H]$. Keep in mind that $k[H^{\prime}]=k[H]\cap k[\hat{H}]$. Therefore $\fq := f_1k[x,y] \cap k[H^{\prime}] = f_1k[H^{\prime}]$. So $$ H_{f,g}^2(k[H^{\prime}])_{\fq} = H_{f_1}^2(k[H^{\prime}])_{\fq} =0.$$ This contradiction says that  $\Ht(f,g)k[x,y] =2$.

Thus $f,g$ is a regular sequence in $k[x,y]$. Note that $fk[H^{\prime}]= fk[x,y] \cap k[H^{\prime}]$. So $f,g$ is a regular sequence in $k[H^{\prime}]$.
\end{proof}

  \begin{lemma}
  Let $C \subseteq \mathbb{Z}^2$ be a semigroup isomorph to a full semigroup $ N$ of $H^{\prime}$   such that for each $(a,b) \in H^{\prime}$, there exists $k \in \mathbb{N}$ such that $k(a,b) \in N$. Then $k[C]$ is Cohen-Macaulay in the sense of Hamilton-Marley.
  \end{lemma}

  \begin{proof}
Recall  that $\fm := \{ h \in k[H^{\prime}]|h_{(0,0)}=0\} $ is a maximal ideal of $k[H^{\prime}]$ and take $f,g \in \fm$.  Let $x^iy^j \in k[C]$ where $i>0$ and $j>0$. Let $(l,k) \in C$. We multiply $(l,k)$ by $t$ to obtain $lt>i,kt>j$. Then $$(lt,kt) = (i,j) + (lt-i,kt-j).$$We find $m\in \mathbb{N}$ such that $m(lt-i,kt-j) \in C$. Therefore $$(x^ly^k)^{mt} = (x^{mi}y^{mj})(x^{mlt-mi}y^{mkt-mj}) \in \fp$$ for all $\fp \in \Var(x^iy^j)$. Thus $\Var(x^iy^j)= \fm$.  Therefore $H_{f,g}^2(k[C])_{\fm} = H_{x^iy^j}^2(k[C])_{\fm} =0$. So $f,g$ can't be a parameter sequence.

   Now we assume that $f,g$ is a parameter sequence in $k[C]$ such that $f_{(0,0)} \neq 0$.  One can find easily that  $k[C]$ is direct summand of $k[H^{\prime}]$ and $k[H^{\prime}]$ is integral over $k[C]$. If $\Ht(f,g)k[x,y] =1$ were be the case, then by the reasoning of Lemma \ref{H'}, there should find $f_1 \in k[H^{\prime}]$ with the property $f_1k[x,y] \in \Var_{k[x,y]}(f,g)$. Write $f_1 = c + xyh$ where $0 \neq c \in k , h \in k[x,y]$. Take $i \in \mathbb{N}_0$ be the maximum integer that $h$ divides $(xy)^i$.  Note that $(f_1,hx^{i+2}y^{i+2})k[x,y]$ is proper in $k[x,y]$. Suppose on the contrary that there are $f^{\prime},g^{\prime} \in k[x,y]$ such that $ f^{\prime}(c+xyh)+g^{\prime}(hx^{i+2}y^{i+2}) =1$. This implies that $f^{\prime}= 1/c$ and $xy/c  + g^{\prime}x^{i+2}y^{i+2} =0 $. This is impossible. Hence $$f_1 k[H^{\prime}] \subseteq (f_1,x^{i+2}y^{i+2})k[x,y] \cap k[H^{\prime}]  \subsetneqq k[H^{\prime}].$$ Besides $hx^{i+2}y^{i+2} \notin f_1 k[H]$, $\Ht(f_1k[H^{\prime}])=1$. So $$\Ht(f_1k[H^{\prime}] \cap k[C]) =1 \  \ and \ \  f,g \in f_1k[H^{\prime}] \cap k[C].$$ This  contradiction says that $\Ht(f,g)k[x,y] =2$. Clearly, $f,g$ is a regular sequence in $k[x,y]$.  Look at $(\dag)$ in Lemma \ref{H'}. That is   $fk[x,y] \cap k[\hat{H}] = fk[\hat{H}]$. This implies that $fk[H^{\prime}] = fk[x,y] \cap k[H^{\prime}]$. Therefore $f,g$ is a regular sequence in $k[H^{\prime}]$. By the  purity of the inclusion map $k[C] \to k[H^{\prime}]$, $f,g$ is  a regular sequence in $k[C]$.
  \end{proof}

  \begin{lemma}
  Let $ H_1 := \{(a,b) \in \mathbb{Z}^2|  b \in \mathbb{N}_0,\; and  \  \ if \;a \; is \; negative\; b\neq 0\} \cup \{(0,0)\}$. Then $k[H_1]$ is Cohen-Macaulay in the sense of Hamilton-Marley.
 \end{lemma}

  \begin{proof}
  Let $\fm := \{f \in k[H_1]| f_{(0,0)}=0\}$. Note that $\fm$ is a maximal ideal of $k[H_1]$.
   Let $\sum_{(i,j)\in H_1} a_{i,j}x^iy^j\in \fm$. Then $\sum_{(i,j)\in H_1} a_{i,j}x^iy^j=x(\sum_{(i,j)\in H_1} a_{i,j}x^{i-1}y^j)$.  This says that
      $\fm = xk[H]$. Let $f,g \in \fm$. Then $$H_{f,g}^2(k[H_1])_{\fm} = H_{x}^2(k[H_1])_{\fm} =0.$$ Therefore, $f,g$ can't be a parameter sequence in $k[H_1]$.

      Set $A:= k[x,y,x^{-1}]$ and consider the localization map $k[H_1] \to k[x,y,x^{-1}]$. Let $f,g$ be a parameter sequence in $k[H_1]$ such that $f_{(0,0)} \neq0$. Hence $\Ht_{k[H_1]}(f,g) =2$. Consider the case $(f,g)k[H_1] \cap \{x^n|n\in \mathbb{N}_0\} \neq \emptyset$. Then, for all $\fp \in \Var_{k[H_1]}((f,g)k[H_1])$, we have $x \in \fp$. This implies that $ \Var_{k[H_1]}((f,g)k[H_1]) = \fm$. So $$ H_{f,g}^2(k[H_1])_{\fm} = H_{x}^2(k[H_1])_{\fm} =0.$$ This contradiction implies that $(f,g)k[H_1] \cap \{x^n|n\in \mathbb{N}_0\} =  \emptyset$.  From this, we conclude that $\Ht_A((f,g)A)=2$. Therefore $f,g$ is a regular sequence in $A$.

 We show that $f,g$ is a regular sequence in $k[H_1]$. Suppose $hg = h_1f$ where $h,h_1 \in k[H_1]$. So $h = fv$ for some $v \in A$. We need to show $v \in k[H_1]$. Note that $ \sum_{0\neq(i,j)\in H_1} f_{i,j}x^iy^j=x(\sum_{0\neq(i,j)\in H_1} f_{i,j}x^{i-1}y^j)$. That is  $f = c + xb $ where $c \in k$ and $b \in k[H_1]$. If  $v \in k[H_1]$ were not be the case, then we should have $$v = c_0+ c_1x^{-1}+ \ldots+c_nx^{-n}+ a$$ where $a \in k[H_1]$ and $n>0$. Thus $$ h= ( cc_0+ cc_1x^{-1}+ \ldots+cc_nx^{-n})+ b(c_0x+ c_1+ \ldots+c_nx^{-n+1}) \in k[H_1].$$ But the coefficient of    $x^{-n}$ in $h$ is $cc_n$ which is nonzero and $(-n,0)\notin H_1$. This  is a contradiction. So $v \in k[H_1]$ and this finishes proof.  \end{proof}

  \begin{lemma}
  Let $C \subseteq \mathbb{Z}^2$ be a semigroup  isomorph to a full semigroup $M$ of $H_1$ such that for each $(a,b) \in H_1$, there exists $k \in \mathbb{N}$ such that $k(a,b) \in M$. Then $k[C]$ is Cohen-Macaulay in the sense of Hamilton-Marley.
 \end{lemma}

  \begin{proof}  First recall that $k[C]$ is a direct summand of $k[H_1]$ and $k[H_1]$ is integral over $k[C]$.
    Look at $\fm_1 := \{f \in k[C]| f_{(0,0)}=0\}$ and take
   $f,g \in \fm_1$. Then $\fm_1$ is a maximal ideal of $k[C]$. There exists $k \in \mathbb{N}$ such that $x^k \in k[C]$. It turns out that $\Var(x^k) = \fm_1$. Therefore $$H_{f,g}^2(k[C])_{\fm_1} = H_{x^k}^2(k[C])_{\fm_1} =0.$$ Then $f,g$ isn't a parameter sequence.

 Now suppose $f,g$ is a  parameter sequence and $f_{(0,0)} \neq 0$. Thus $\Ht_{k[C]}(f,g) =2.$ We deduce from this to observe $\Ht_{k[H_1]}(f,g) =2.$ If $(f,g)k[H_1] \cap \{x^n|n\in \mathbb{N}_0\} \neq \emptyset$, then for all $\fp \in \Var_{k[H_1]}((f,g)k[H_1])$, we have $x \in \fp$. This implies that $$ \Var_{k[H_1]}((f,g)k[H_1]) = \fm.$$ But $f \notin \fm $. In view of this contradiction, $$(f,g)k[H_1] \cap \{x^n|n\in \mathbb{N}_0\} =  \emptyset.$$ Now conclude  that $\Ht_A((f,g)A)=2$. Therefore $f,g$ is a regular sequence in $A$. So $f,g$ is a regular sequence in $k[H_1]$. In view of the purity, $f,g$ is a regular sequence in $k[C]$.
  \end{proof}

  \begin{lemma}\label{h2}
  Let  $ H_2 := \{(a,b) \in \mathbb{Z}^2|  b \in \mathbb{N}\} \cup \{(0,0)\}$. Then $k[H_2]$ is Cohen-Macaulay in the sense of Hamilton-Marley.
 \end{lemma}

  \begin{proof}
   Let $\fm := \{f \in k[H_2]| f_{(0,0)}=0\}$ and take $f,g\in \fm$. Then $\fm$ is a maximal ideal of $k[H_2]$.
    Let $x^ay^b \in \fm$. Then $(x^{a}y^{b})^2= (xy)(x^{2a-1}y^{2b-1}) \in \fp$ for all $\fp \in \Var_{k[H_2]}(xyk[H_2])$.
   This implies that $\Var_{k[H_2]}(xyk[H_2]) = \fm$. Therefore, $$ H_{f,g}^2(k[H_2])_{\fm} = H_{xy}^2(k[H_2])_{\fm} =0.$$ So $f,g$ can't be a parameter sequence in $k[H_2]$.

   Let $f,g$ be a parameter sequence in $k[H_2]$ such that $f_{(0,0)} \neq0$. Then $\Ht_{k[H_2]}(f,g) =2$. Consider the ring $A:= k[x,y,x^{-1}]$. Suppose on the contrary that $\Ht_A((f,g)A) =1$. Since $A$ is a unique factorization domain, there is a prime ideal $\fp$ in $A$ such that $\fp$ is minimal over $(f,g)A$ and $\fp = f_1 A$ for some  $f_1 \in A$. We show that one can choose $f_1 \in k[H_2]$. Look at $f_2 \in k[x,x^{-1}]$ and $ f_3 \in k[H_2]$ with the properties $f_1 = f_2 +f_3$ and $(f_3)_{(0,0)} =0$. There is $h \in A$ such that $f_1h = f$. We choose $h_2 \in k[x,x^{-1}]$ and $ h_3 \in k[H_2]$ such that $h = h_2 +h_3$ and $(h_3)_{(0,0)} =0$. Keep in mind that $f_3h_1$, $f_3h_3$ and $f_2h_3$ are in $\fm$.
   Also, recall that $f_2, h_2\in k[x,x^{-1}]$. Since monomials in $k[H_2]$ are not involved on $x^{\pm n}$,
      then $f_2h_2 = f_{(0,0)} \in k[H_2]$. Remark that  $f_2$ must be an invertible element in $k[x,x^{-1}]$. If $f_2 \in k$, then $f_1 \in k[H_2]$. Else if $f_2= ax^n$ where $n \in \mathbb{Z}$ and $a\in k$, then $x^{-n}f_1 \in k[H_2]$. Thus $f_1 A = x^{-n}f_1 A$.
   Replacing  $f_1$ by $x^{-n}f_1$, we assume that $f_1 \in k[H_2]$. Clearly,  $$f_1 k[H_2] \subseteq f_1A \cap k[H_2].$$ We show that this is an equality.  Write $f_1 = c + f_2$  for $c \in k$ and $f_2 \in \fm$.
   Also, take $h \in A$ and write it as $h = h_1 + h_2 $ where $h_1 \in k[x,x^{-1}]$ and $h_2 \in \fm$.
   Assume that $f_1 h \in k[H_2]$. Hence $ch_1 \in k[H_2]$. Thus $ h_1 \in k$ and $h \in k[H_2]$. Therefore, $$f_1 k[H_2] =  f_1A \cap k[H_2].$$ So $$f_1 k[H_2] \in \Var_{k[H_2]}((f,g)k[H_2]).$$Then $$ H_{f,g}^2(k[H_2])_{f_1k[H_2]} = H_{f_1}^2(k[H_2])_{f_1k[H_2]} =0.$$ This contradiction says that $\Ht_A((f,g)A)= 2$. Thus $f,g$ is a regular sequence in $A$. Since $fA \cap k[H_2] = fk[H_2]$,  $f,g$ is a regular sequence in $k[H_2]$.
   \end{proof}

 \begin{lemma}
  Let $C \subseteq \mathbb{Z}^2$ be a semigroup isomorph to a  full semigroup $M$ of $H_2$ and for each $(a,b) \in H_2$  there exists $k \in \mathbb{N}$ such that $k(a,b) \in M$. Then $k[C]$ is Cohen-Macaulay in the sense of Hamilton-Marley.
 \end{lemma}

 \begin{proof}
  Clearly,  $k[C]$ is a direct summand of $k[H_2]$ and $k[H_2]$ is integral over $k[C]$.
   Look at the maximal ideal $\fm_1 := \{f \in k[C]| f_{(0,0)}=0\}$
  and take $f,g \in \fm_1$.   Let $i \in \mathbb{Z}$ and $j \in \mathbb{N}$ be such that $x^iy^j\in k[C]$. An easy computation shows that $\Var_{k[C]}( (x^i y^j)k[C]) = \fm_1$. Therefore, $ H_{f,g}^2(k[C])_{\fm} = H_{x^{i}y^{j}}^2(k[H_2])_{\fm} =0.$ So $f,g$ can't be a parameter sequence in $k[C]$.

   Let $f,g$ be a parameter sequence in $k[C]$ such that $f_{(0,0)} \neq0$. Recall that $\Ht_{k[C]}(f,g) =2$ and  $\Ht_{k[H_2]}(f,g) =2$.
   Set $A:= k[x,y,x^{-1}]$. If $\Ht_A((f,g)A) = 1$, then  by the reasoning of Lemma \ref{h2}, there exists $f_1 \in k[H_2]$ such that $f_1k[H_2] \in \Var_{k[H_2]}((f,g)k[H_2])$. Write $f_1 = c + f_2$, where $c \in k$ and $f_2 \in \fm$. Assume that the maximum degree of $y$ in $f_2$ is $n$. Then $y^{n+1} \notin f_1k[H_2]$ and $(f_1,y^{n+1})k[H_2]$ is  a proper ideal of $k[H_2]$. So $\Ht_{k[H_2]}(f_1k[H_2])=1$. This contradiction implies that $\Ht_A((f,g)A) = 2$. Thus
$f,g$ is a regular sequence in $A$. So $f,g$ is a regular sequence in $k[H_2]$. By the purity,  $f,g$ is a regular sequence in $k[C]$.
  \end{proof}

$$\textmd{Subsection 8.4: Cohen-Macaulayness of quasi rational plane cones}$$

The following is our main result of this section.

\begin{theorem} \label{ggg}Let $C$ be the normal submonoid of $\mathbb{Z}^2$ defined by Discussion \ref{disdef}.
Then any parameter sequence of $k[C]$ is a regular sequence.
\end{theorem}

\begin{proof}
Without loss of the generality we can assume that $C$ is positive.
By the Proposition \ref{ch}, $C$ is full in  $\{H,H^{\prime},H_1,H_2\}$ and  the extension is integral. We showed
by lemmas of subsection 8.3 that all of these are Cohen-Macaulay  in the sense of Hamilton-Marley. This completes the proof.
\end{proof}

$$\textmd{Subsection 8.5: Toward the classification}$$

\begin{discussion}\label{dis}
Let $C\subseteq \mathbb{Z}^2$ be a positive normal semigroup such that $C$ is not finitely generated and assume that $C-C = \mathbb{Z}^2$.  \begin{enumerate}
\item[(i)] There are two elements $a,b \in C$  linearly independent over $\mathbb{Q}$ and $\mathbb{Q}C = \mathbb{Q}^2 = \mathbb{Q}a +\mathbb{Q}b$. Take the integer $t \in \mathbb{N}$ be such that  $t(1,0), t(0,1) \in \mathbb{Z}a + \mathbb{Z}b $.   Define the  linear map $\varphi: \mathbb{Q}^2 \to \mathbb{Q}^2$ via the assignments  $\varphi(a)=t(1,0)$  and $\varphi(b)=t(0,1)$. Note that $\varphi$ is an isomorphism.
Thus, $\varphi(C)$ is normal and positive.

     \textbf{Claim.} $\varphi(C) \subseteq \mathbb{Z}^2$. \\Indeed, let $c:=(m,n)\in C$. Take the integers
     $\{m',m'',n',n''\}$ be such that $$tc=tm(1,0)+tn(0,1)=m(m'a+m''b)+n(n'a+n''b).$$ Hence $\varphi(tc)=(mm'+nn')t(0,1)+(nn'+mm'')t(1,0)$. So $$\varphi(c)=(mm'+nn')(0,1)+(nn'+mm'')(1,0)\in\mathbb{Z}^2.$$
\item[(ii)] If $P \in \mathbb{N}^2$, then $tP \in \varphi(C)$. This follows by $t(1,0),t(0,1) \in \varphi(C)$.
\item[(iii)] Let $P$ be a point in the third quarter of the plane. Then $ P \notin \varphi(C)$,   because of (ii) and  the positivity of $\varphi(C)$.
\end{enumerate}
\end{discussion}

\begin{notation} Denote the origin  by $o$.

Denote the half line beginning by $p$ and path through the $q$ by $\overrightarrow{pq}$.
\end{notation}

\begin{lemma}\label{ch2}
Adopt the notation of Discussion \ref{dis} and let $i=1,2$. Then there are half-lines $l_i$ in the $2i$-th quarter both cross through  the origin with the following properties:
 \begin{enumerate}
\item[(1)] $\varphi(C) \subseteq \conv(l_1,l_2)\cap\mathbb{Z}^2$.
\item[(2)] Let $P$  be in the interior of $\conv(l_1,l_2)$ or probably on the one of these half-lines. Then $tP \in \varphi(C)$ for some $t\in \mathbb{N}$.
\item[(3)]  The anti-clock angel from $l_2$ to $l_1$ is either $\pi$  radian or less than $\pi$ radian.
\end{enumerate}
\end{lemma}

\begin{proof}Define $\alpha:=\sup\{\angle(\overrightarrow{ox},\overrightarrow{o(1,0)})|x\in\varphi(C)\textit{ is in the forth quarter}\},$
and take $l_2$ be the half line path through the origin such that $\alpha=\angle(l_2,\overrightarrow{o(1,0)})$.
Also, define $$\beta:=\inf\{\angle(\overrightarrow{ox},\overrightarrow{o(-1,0)})|x\in\varphi(C)\textit{ is in the second quarter}\},$$
and take $l_1$ be the half line path through the origin such that $\beta=\angle(l_1,\overrightarrow{o
(-1,0)})$.

Let $t$ be as Discussion \ref{dis}. Let $p_1:=(a_1,b_1) \in  \varphi(C)$  be in the second quarter and $(a_2,b_2) \in \overrightarrow{op_1}\cap\mathbb{Z}^2$. Look at  the following observations.

(i): One has $t(a_2,b_2) \in \varphi(C)$.
Indeed,  first note that $t(a_2,b_2)$ is in the group that $\varphi(C)$ generates. Because, $ t(a_2,b_2)\in \mathbb{Z}\varphi(a) + \mathbb{Z}\varphi(b)$. Also, there exists a positive rational number $q =m/n$ such that $q t a_2 = a_1$ and $q t b_2 =b_1$. So $$m t(a_2,b_2) = n(a_1,b_1)\in  \varphi(C).$$ Since  $\varphi(C)$ is normal, $t(a_2,b_2) \in \varphi(C)$.

(i)': Suppose that $q_1:=(c_1,d_1) \in  \varphi(C)$  is in the forth quarter and $(c_2,d_2) \in \overrightarrow{oq_1}\cap\mathbb{Z}^2$.
Then by the symmetry of the above item, $t(c_2,d_2) \in \varphi(C)$.

(ii): Suppose $(a_3,b_3) \in \mathbb{Z}^2$ is in the interior of $ \conv(\overrightarrow{o(0,1)},\overrightarrow{o(a_2,b_2)})$.  Then
 $ t(a_3,b_3) \in \varphi(C)$. Indeed, first
note that
$( a_3/a_2) . b_2 <  b_3$. Set $ 0 < q_1:=  b_3-(  a_3/a_2) . b_2$ and $q_2:=  a_3 /a_2$. Hence $t (a_3,b_3) = tq_2(a_2,b_2)+ q_1(0,t)$. Also $$t (a_2,b_2)  \in \mathbb{Z}\varphi(a)+ \mathbb{Z}\varphi(b)\subseteq\widehat{\varphi(C)}_{\varphi(C)-\varphi(C)}.$$ So
$t (a_3,b_3)\in\widehat{\varphi(C)}_{\varphi(C)-\varphi(C)}$, since $\widehat{\varphi(C)}_{\varphi(C)-\varphi(C)}$ is a semigroup.
The normality of $\varphi(C)$ implies that $ t(a_3,b_3) \in \varphi(C)$.

(ii)': Having $(c_2,d_2)$ as in the item (i)' and  suppose that $(c_3,d_3) \in \mathbb{Z}^2$ is in the interior of $ \conv(\overrightarrow{o(1,0)},\overrightarrow{o(c_2,d_2)})$. Then by the symmetry of the above item,
 $ t(a_3,b_3) \in \varphi(C)$.\\
We are in the position to  prove the Lemma.

\begin{enumerate}
\item[(1)] $\varphi(C)\subseteq \conv(l_1,l_2)\cap\mathbb{Z}^2$. This is clear by definition of $l_i$ and  Discussion \ref{dis}(iii).
\item[(2)] Let $P$  be in the interior of $\conv(l_1,l_2)$ or probably on the one of these half-lines. Then $tP \in \varphi(C)$ for some $t\in \mathbb{N}$. Indeed, if $P$ is in the first quadrant, this follows by Discussion \ref{dis} (ii). If $P$ is in the second quadrant, this follows by the above observations (i) and (ii). If $P$ is in the forth quadrant, see (i)' and (ii)'.
\item[(3)]$\angle(l_2,l_1)$ is less or equal than $\pi$ radian, because of the positivity of $\varphi(C)$ and (2).
\end{enumerate}
\end{proof}

\begin{lemma}\label{full}
Adopt the above notation.  Let $M$ be the set consists of all $P$  in the interior of $\conv(l_1,l_2)$ or probably on the one of these half-lines. Then
$\varphi(C)$  is full in $M$.
\end{lemma}

\begin{proof} Take $h,h^\prime \in \varphi(C)$ with the property that
$h-h^\prime \in M$. Look at  $P:=h-h^\prime$. Due to (2) in Lemma \ref{ch2}, we have
$tP \in \varphi(C)$  for some $t\in \mathbb{N}$. By the normality of $\varphi(C)$,
$P \in\varphi(C)$. In view of Definition, $\varphi(C)$  is full in $M$.
\end{proof}

\begin{corollary}\label{ch3}
Adopt the above notation.
Then  $\varphi(C)$ intersects at most one of the $\{l_1,l_2\}$.
\end{corollary}

\begin{proof}  Suppose on the contrary that  $\varphi(C)$ intersects nontrivially with $l_1$ and $l_2$. Then by the reasoning of Lemma \ref{full}, $\varphi(C)$ is full in $\conv(l_1,l_2)$. In view of
\cite[Corollary 2.10]{BG}, the integer points of  $\conv(l_1,l_2)$ is a finitely generated semigroup.
Recall that $k[\varphi(C)]\subseteq K[\conv(l_1,l_2)]$ is pure, because of the fullness. Thus
$K[\varphi(C)]$ is affine. So $\varphi(C)$ is finitely generated. This is a contradiction.
\end{proof}


\end{document}